\documentclass[11pt]{amsart}
\usepackage{dsfont}
\usepackage[euler-digits]{eulervm}
\usepackage{graphicx, color}

\newcommand{\Flat}{\mathcal T}

\title[Bifurcation of triply periodic minimal surfaces]{On bifurcation and local rigidity\\ of triply periodic minimal surfaces in $\mathds R^3$}
\author{M. Koiso \and P. Piccione \and T. Shoda}
\address{%
\begin{tabular}{ll}
Institute of Mathematics for Industry & Departamento de Matem\'atica \\
Kyushu University &  Universidade de S\~ao Paulo \\
744, Motooka Nishi-ku  &  Rua do Mat\~ao
1010 \\
FUKUOKA 819-0395 &  CEP 05508-900, S\~ao Paulo, SP\\
JAPAN & BRAZIL\\
\texttt{koiso@math.kyushu-u.ac.jp} &  \texttt{piccione.p@gmail.com}\\
\end{tabular}
\bigskip
\hfill\break\hfill\indent
\begin{tabular}{lll}
Department of Mathematics&&\\
Faculty of Culture and Education&&\\
Saga University&&\\
Saga, SAGA 840&&\\
JAPAN &&\\ \texttt{tshoda@cc.saga-u.ac.jp}\\
\end{tabular}
}
\date{November 24, 2014}
\subjclass[2010]{Primary: 53A10, 58J55, 58E12; Secondary: 35J62.}
\thanks{The first author is partially supported by
Grant-in-Aid for Scientific Research (B) No. 25287012
of the Japan Society for the Promotion of Science,
and the Kyushu University Interdisciplinary Programs in Education and Projects in Research Development.
The second author is partially supported by Fapesp and CNPq, Brazil.
The third author is partially supported by JSPS Grant-in-Aid
 for Young Scientists (B) 24740047.}

\keywords{triply periodic minimal surfaces, H-family, rPD-family, tP-family, tD-family, tCLP-family, bifurcation theory}

\theoremstyle{plain}
\newtheorem{proposition}{Proposition}[section]
\newtheorem{definition}[proposition]{Definition}
\newtheorem{lemma}[proposition]{Lemma}
\newtheorem{theorem}[proposition]{Theorem}

\newtheorem*{biftheorem}{Bifurcation Theorem for Fredholm Operators}
\newtheorem*{mainA}{Theorem A}
\newtheorem*{mainB}{Theorem B}
\newtheorem*{mainC}{Theorem C}
\swapnumbers
\theoremstyle{remark}
\newtheorem{remark}[proposition]{Remark}
\newtheorem{example}[proposition]{Example}

\begin{document}

\maketitle

\begin{abstract}
We use bifurcation theory to determine the existence of infinitely many new examples of triply periodic minimal surfaces
in $\mathds R^3$. These new examples form branches issuing from the H-family, the rPD-family, the tP-family, and the tD-family, that converge to some degenerate embedding of the families. As to nondegenerate triply periodic minimal surfaces,
we prove a perturbation result using an equivariant implicit function theorem.
\end{abstract}

\section{Introduction}
Construction of examples and classification of triply periodic minimal surfaces in $\mathds R^3$ constitute a very active field of research in Differential Geometry. Such surfaces correspond, via universal covering, to minimal embeddings of closed orientable surfaces into a flat torus $\mathds T^3$. As originally conjectured by Meeks, see \cite{mee90}, it is known that every closed surface of genus greater than or equal to $3$ admits a minimal embedding in a flat torus, see \cite{Trai08}.
Interestingly enough, triply periodic minimal surfaces appear naturally in several applied sciences, including
physics, chemistry, and crystallography, see for instance \cite{AndHydLid, FisKoc87, FogHyd99, SchNes91}. G. E. Schr\"oder-Turk, A. Fogden,
and S. T. Hyde \cite{SchFogHyd06} studied one-parameter families of triply periodic minimal
surfaces in $\mathds R^3$. These families which are called H-family, rPD-family,
tP-family, tD-family, and tCLP-family, contain many classical examples
(Schwarz P-surface, Schwarz D-surface, Schwarz H-surface, and Schwarz CLP-surface).

The purpose of this paper is to give an abstract proof of the existence of infinitely many new examples
of triply periodic minimal embeddings of a closed orientable surface $\Sigma$ of genus $3$ into some flat $3$-torus (Theorems A and \ref{thm:mainbiftheorem}),
using perturbation techniques, and in particular bifurcation theory. The way to a development of a bifurcation theory  in this context was paved
by some recent numerical results of N. Ejiri and T. Shoda, see \cite{E-S}, who use a finite dimensional approach to compute
the nullity and the Morse index for the above families of minimal embeddings.

Let us describe more precisely the result discussed in this paper. Given a one-parameter family $\left]a_0-\varepsilon,a_0+\varepsilon\right[\ni a\mapsto g_a$ of (unit volume)
flat metrics in $\mathds T^3$, and a one-parameter family $a\mapsto x_a$ of $g_a$-minimal embeddings
$x_a:\Sigma\to\mathds T^3$, then a \emph{bifurcating branch of minimal embeddings converging to $x_{a_0}$} for
the family $(x_a)_a$ consists of:
\begin{itemize}
\item a sequence $a_n$ tending to $a_0$ as $n\to\infty$;
\item a sequence $x_n\colon\Sigma\to\mathds T^3$ of embeddings, where $x_n$ is $g_{a_n}$-minimal for all $n$,
\end{itemize}
such that
\begin{itemize}
\item[i.] $\lim\limits_{n\to\infty}x_n=x_{a_0}$ in some suitable $C^k$-topology, with $k\ge2$;
\item[ii.]  $x_n$ is not \emph{congruent} to $x_{a_n}$ for all $n$.
\end{itemize}
Let us also recall that two embeddings $y_1,y_2:\Sigma\to\mathds T^3$ are congruent if one is obtained from the other by a change of parameterization of $\Sigma$ and by a translation of $\mathds T^3$, i.e., if there exists a diffeomorphism
$\psi$ of $\Sigma$ and $\mathfrak t\in\mathds T^3$ such that $y_1=y_2\circ\psi+\mathfrak t$.

In the situation above, we say that $a_0$ is a \emph{bifurcation instant} for the family $(x_a)_a$. An important related notion is that of \emph{local rigidity} for a family of minimal embeddings, which will be discussed in Section~\ref{sec:locrigidity}, based on the notion of \emph{equivariant nondegeneracy}, see Definition~\ref{thm:defequivnondegeneracy}.

We will recall in Section~\ref{sec:applications} the definitions of the H-family, the rPD-family, the tP-family, the tD-family and the tCLP-family of triply periodic minimal embeddings. The precise statement of the main result proved in the present paper is the following:
\begin{mainA}
There is one bifurcation instant for the H-family,
 and two bifurcation instants
for the rPD-family, the tP-family, and the tD-family.
\end{mainA}
Estimates of the above bifurcation instants are obtained using the numerical results of \cite{E-S}. The proof of Theorem~A is based on some techniques of
equivariant bifurcation theory (see Theorem~\ref{thm:mainbiftheorem}), and it also employs some recent results obtained in \cite{E-S} for the computation of
Morse index and nullity of minimal embeddings in the above families.

A few remarks on our proof are in order.

The first important issue is the question of \emph{degeneracy} of the minimal embeddings in the above families, caused
by the symmetries of the ambient space $\mathds T^3$. Every minimal embedding admits a three dimensional space of (non trivial) Jacobi fields, coming from the Killing fields of the ambient, which implies that each one of the embeddings is a degenerate critical point of the area functional.
Recall that standard variational bifurcation assumptions require nondegeneracy at the endpoints of the path, which fails
to hold in this situation. The central technical part of the paper consists in the
 construction of an alternative functional
framework, suitable to handle such degenerate situation. To this aim, we use an idea originally introduced by Kapouleas to prove
an implicit function theorem for constant mean curvature embeddings, which consists in considering a ``perturbed'' mean curvature function $\widetilde{\mathcal H}$ of an embedding, see Section~\ref{sec:framework} for details.
Such a function vanishes identically exactly at minimal embeddings (Proposition~\ref{thm:kapouleas}), nevertheless maintains its differential
surjective at possibly degenerate minimal embeddings, provided that the degeneracy arises exclusively from the ambient symmetries (Proposition~\ref{eq:FredholmnessT}). We call \emph{equivariantly nondegenerate} a minimal embedding along which
every Jacobi field arises from a Killing field of the ambient space. A direct application of an equivariant implicit function
theorem to the equation $\widetilde{\mathcal H}=0$ gives a local rigidity result for each of the above families of
minimal embeddings in neighborhoods of equivariantly nondegenerate embeddings (Theorem~\ref{thm:IFT}). As a corollary of Theorem~\ref{thm:IFT}, we prove the following:
\begin{mainB}\label{thm:localrigidity}
Every minimal surface in the tCLP-family, as well as any surfaces in the H-family, the tP-family, the tD-family, and the rPD-family whose nullity is equal to $3$, belongs to a (unique up to homotheties) smooth
locally rigid $5$-parameter family of pairwise non-homothetic triply periodic minimal surfaces.
\end{mainB}
\begin{remark}\label{thm:remMeeks}
It is interesting to observe that W. Meeks proved in \cite[Theorem 7.1]{mee90} that
every triply periodic minimal surface of genus $3$ for which the ramified values of the Gauss map
consist of $4$ antipodal pairs in the $2$-sphere, belongs to a real $5$-dimensional family of triply
periodic minimal surfaces of genus $3$.
Every surface in the  tCLP-family, the tP-family, the tD-family, and the rPD family
belongs to the class of surfaces to which \cite[Theorem 7.1]{mee90} applies.
Thus, for each surface $x_0:\Sigma\to\mathds T^3$ with nullity three in these families, the $5$-parameter family of triply periodic minimal surfaces that contains $x_0$ given in Theorem~B coincides with Meeks' family.
For these examples, the new information provided by our results, besides a different approach to the proof, is the local rigidity property of these families around surfaces with nullity equal to $3$, and the lack of local rigidity around surfaces with nullity larger than $3$.
\end{remark}
\begin{remark}\label{thm:remmoregenerally}
The existence of a $5$-parameter family of triply periodic minimal surfaces containing
a given one is obtained in Theorem~\ref{thm:IFT}, more generally,
near each embedded triply periodic minimal surface with genus greater than one and with nullity equal to three.
The classical example of Schoen's gyroid
does not satisfy the assumptions of \cite[Theorem 7.1]{mee90}
(see \cite[Remark 7.2]{mee90}).
However, since it belongs to the associate family of $P$ and $D$-surfaces (Example \ref{exa:PDG}),
it has nullity equal to $3$, and Theorem~\ref{thm:IFT} applies in this situation.
\end{remark}

\begin{mainC}\label{thm:localrigidityC}
Schoen's gyroid belongs to a (unique up to homotheties) smooth
locally rigid $5$-parameter family of pairwise non-homothetic triply periodic minimal surfaces.
\end{mainC}
We remark that the existence of a family of deformations for the triply periodic minimal surfaces considered in Theorem~B and Theorem~C above can also be deduced from Ejiri's results in \cite{eji2013}.
Actually, while we obtain the results via an equivariant implicit function theorem,
\cite{eji2013} uses a different approach, and it gives a more explicit description of the deformation space.
\smallskip

As an undesired drawback in our bifurcation setup, we need to observe that the PDE: \[\widetilde{\mathcal H}=0,\]
defined in the space of ``unparameterized embeddings'' of $\Sigma$ into $\mathds T^3$, is \emph{not} variational,
i.e., it is not the Euler--Lagrange equation of some variational problem (recall that the standard mean curvature
function is the gradient of the area functional). This entails that, in order to carry out our project,
we have to resort to more general bifurcation theory for Fredholm operators (see Appendix~\ref{sub:bifresult}),
which provides results somewhat weaker than variational bifurcation theorems. More specifically, rather than
the classical ``jump of Morse index'' assumption, in the non variational case one has to postulate the less
general (and intuitive) ``odd crossing number'' condition for the eigenvalues of the linearized problem.
In particular, we can only infer the existence of bifurcating branches at those instants at which the jump
of the Morse index is an \emph{odd} integer, leaving undecided the existence of a bifurcating branch at a certain
degeneracy instant ($a_1\approx0.71479$) for the H-family, where the jump of Morse index is equal to $2$, see Section~\ref{sec:applications}.

After establishing our general bifurcation result for minimal embeddings into flat $3$-tori (Theorem~\ref{thm:mainbiftheorem}), a proof of Theorem~A is obtained readily (Section~\ref{sec:applications}) by applying the results of \cite{E-S}, where the nullity and the Morse index of the minimal embeddings in the given families are computed.

In Section~\ref{sec:remarks}, we will present an analysis of the type of bifurcation occurring at the bifurcation instants determined in Theorem~A. We will show that for the $H$-family, the degeneracy instant corresponds to a \emph{transcritical bifurcation}, which does not produce \emph{essentially new} triply periodic minimal surfaces. This is due to the fact that, around the degenerate instant where bifurcation occurs, the homothety class of the flat metric on the torus does not depend bijectively on the parameter of the family,
see Remark~\ref{thm:remgenuinebifurcation}.
The same situation occurs at one of the two bifurcation instants of the families rPD, tP and tD.
On the other hand, the second bifurcation instant of each of these three families is \emph{genuine}, in the sense that the bifurcation branch that issues from these instants consists of triply periodic minimal surfaces that are not homothetic to any other member of the family.

As a final remark, we would like to observe that a full-fledged bifurcation theory for triply periodic minimal
surfaces in $\mathds R^3$ will require a further development of the results exposed here. In first place, it would be interesting to extend the existence result to degeneracy instants corresponding to \emph{even} jumps of the Morse index, which is very likely a matter of applying finer bifurcation results.
The second point would be to study the topology of the bifurcating branches, like connectedness, cardinality, regularity, etc., which ultimately depends on the behavior of the eigenvalues of the Jacobi operator near zero (derivative, transversal crossing). An important question to assess is establishing the pitchfork picture of the
bifurcation set, and the
 stability/instability of minimal surfaces in the bifurcating branches.
Note that a triply periodic minimal surface that divides ${\mathds T}^3$ into two parts  is said to be stable if the second variation of the area is nonnegative for all volume-preserving variations as a compact surface in  ${\mathds T}^3$ with the corresponding metric.
Here, by volume it is meant the volume of each part of  ${\mathds T}^3$ divided by the surface.
For instance, the Schwarz $P$ and $D$ surfaces and Schoen's gyroid are stable (Ross \cite{Ross1992}).

Finally, it would be very interesting to study the geometry of the new triply periodic minimal surfaces in the bifurcating branches issuing from the genuine bifurcation instants along the  rPD, the tP and the tD family.
These topics constitute the subject of an ongoing research project by the authors.

\bigskip

\noindent\textbf{Acknowledgement:}
Most of the pictures of surfaces in this paper were originally drawn by Professor Shoichi Fujimori (Okayama University, Japan). The authors express their gratitude to him.

\section{The functional setup}
\subsection{Notations and terminology}
\label{sub:notations}
We will denote by $\Lambda$ a generic lattice in $\mathds R^3$. The quotient $\mathds R^3/\Lambda$ is
diffeomorphic to the $3$-torus $\mathds T^3$, the quotient map $\mathds R^3\to\mathds R^3/\Lambda$ will be denoted by $\pi_\Lambda$ and the induced flat metric will be denoted by $g_\Lambda$.
The identity connected component of the isometry group of $(\mathds R^3/\Lambda,g_\Lambda)$ consists
of translations $\mathfrak t\mapsto\mathfrak t+\mathfrak t_0$, $\mathfrak t,\mathfrak t_0\in\mathds R^3/\Lambda$.

The symbol $\Flat(\mathds T^3)$ will denote\footnote{%
We have a surjective map from $\mathrm{GL}(3)=\mathrm{GL}(3,\mathds R)$ to the set of lattices of $\mathds R^3$:
given $A\in\mathrm{GL}(3)$, one associates the lattice $\Lambda_A=\mathrm{span}_\mathds Z\{Ae_1,Ae_2,Ae_3\}$, where $e_1,e_2,e_3$ is the canonical basis of $\mathds R^3$.
Given $A,A'\in\mathrm{GL}(3)$, then the lattices $\Lambda_A$ and $\Lambda_{A'}$ are isometric
if and only if there exists $U\in\mathrm O(3)$ such that $A'=UA$. Thus, $\Flat(\mathds T^3)$ is identified with
the quotient space $\mathrm O(3)\backslash \mathrm{GL}(3)$.  This is a $6$-dimensional manifold;
the set $\Flat_1(\mathds T^3)$ of isometry classes of lattices having volume $1$ has dimension equal to $5$.} the set of flat metrics on $\mathds T^3$ modulo isometries,
or, equivalently, the set of isometry classes of lattices of $\mathds R^3$.
The isometry class of a flat metric $g$ will be denoted by $[g]$, and the isometry class of a lattice
$\Lambda$ will be denoted by $[\Lambda]$. The \emph{volume} of a lattice $\Lambda$ is the volume
of the metric $g_\Lambda$; by $\Flat_1(\mathds T^3)$ we will denote the isometry classes of unit volume
lattices of $\mathds R^3$.

Let $\Sigma$ be a closed surface; in our main applications, $\Sigma$ will be a closed orientable surface of genus $3$.
For $k\in\mathds N\bigcup\{0\}$ and $\alpha\in\left]0,1\right[$,
the symbol $C^{k,\alpha}(\Sigma)$ will denote the Banach space of $C^{k,\alpha}$ real functions on $\Sigma$.

Let $\Lambda_0$ be a fixed lattice of $\mathds R^3$, $g_0=g_{\Lambda_0}$ be the corresponding flat metric
on $\mathds T^3$, and let us assume that $x_0\colon\Sigma\to\mathds T^3$ is a fixed $g_0$-minimal embedding,
which is transversally oriented. Given $[\Lambda]$ sufficiently close to $[\Lambda_0]$ and $\varphi\in C^{2,\alpha}$ near $0$, let us denote
by $x_{\varphi,\Lambda}\colon\Sigma\to\mathds T^3$ the embedding:
\[\phantom{,\quad p\in\Sigma,}x_{\varphi,\Lambda}(p)=\exp^{g_\Lambda}_{x_0(p)}\big(\varphi(p)\cdot\vec n^{g_\Lambda}_{x_0(p)}\big),\quad p\in\Sigma,\]
where $\exp^{g_\Lambda}$ is the exponential map of the metric $g_\Lambda$, and $\vec n^{g_\Lambda}_{x_0}$
is the positively oriented $g_\Lambda$-unit normal vector along $x_0$. It is well known that, for $\Lambda$ fixed,
the map $\varphi\mapsto x_{\varphi,\Lambda}$ gives a bijection between a neighborhood of $0$ in $C^{2,\alpha}(\Sigma)$ and a neighborhood of $x_0$ in the space of \emph{unparameterized embeddings}
(i.e., embeddings modulo reparameterizations) of $\Sigma$ into $\mathds T^3$. Details of this construction can be found, for instance, in reference \cite{AliPic10}.

Finally, for fixed $\Lambda$ and $i=1,2,3$, set $K_i^\Lambda=(\pi_\Lambda)_*(e_i)$,
where $\{e_1,e_2,e_3\}$ is the canonical basis of $\mathds R^3$.
The $K_i^\Lambda$, $i=1,2,3$, form a basis of Killing vector fields
of $(\mathds T^3,g_\Lambda)$. For $\varphi\in C^{2,\alpha}$ near $0$ and $i=1,2,3$, define $f_i^{\varphi,\Lambda}\colon\Sigma\to\mathds R$ by:
\[f_i^{\varphi,\Lambda}=g_\Lambda\big(K_i^\Lambda,\vec n^{g_\Lambda}_{x_{\varphi,\Lambda}}\big).\]
Here, $\vec n^{g_\Lambda}_{x_{\varphi,\Lambda}}$ denotes the $g_\Lambda$-unit normal field along
the embedding $x_{\varphi,\Lambda}$. We will also denote by $K_i^{\varphi,\Lambda}$ the vector fields
on $\Sigma$ obtained by $g_\Lambda$-orthogonal projection to $x_{\varphi,\Lambda}$ of $K_i^\Lambda$:
\begin{equation}\label{eq:Kilambdaphi}
K_i^{\varphi,\Lambda}=K_i^\Lambda-f_i^{\varphi,\Lambda}\cdot\vec n^{g_\Lambda}_{x_{\varphi,\Lambda}}.
\end{equation}
\subsection{The functional framework}
\label{sec:framework}
Given $[\Lambda]\in\Flat(\mathds T^3)$ and an embedding $x:\Sigma\to\mathds T^3$, let us denote by $\mathcal H^\Lambda(x)\colon\Sigma\to\mathds R$ the mean curvature function of the embedding $x$ relative to the
metric $g_\Lambda$. Let us consider the function:
\[\widetilde{\mathcal H}\colon\mathfrak U_0\times\mathds R^3\times\mathfrak V_0\longrightarrow C^{0,\alpha}(\Sigma),\]
where $\mathfrak U_0$ is a neighborhood of $0$ in the Banach space $C^{2,\alpha}(\Sigma)$ and
$\mathfrak V_0$ is a neighborhood of $[\Lambda_0]\in\Flat(\mathds T^3)$, defined by:
\begin{equation}\label{eq:centralmap}
\widetilde{\mathcal H}\big(\varphi,a_1,a_2,a_3,[\Lambda]\big)=\mathcal H^\Lambda(x_{\varphi,\Lambda})+\sum_{i=1}^3a_if_i^{\varphi,\Lambda}.
\end{equation}
For $[\Lambda]\in\Flat(\mathds T^3)$, we will also use the notation:
\[\widetilde{\mathcal H}_\Lambda\colon\mathfrak U_0\times\mathds R^3\longrightarrow C^{0,\alpha}(\Sigma)\]
for the map:
\begin{equation}\label{eq:centralmap2}
\widetilde{\mathcal H}_\Lambda(\varphi,a_1,a_2,a_3)=\widetilde{\mathcal H}\big(\varphi,a_1,a_2,a_3,[\Lambda]\big).
\end{equation}
The following result is based on an idea of N. Kapouleas  \cite{Kap1,Kap2},
which was then also employed by R. Mazzeo, F. Pacard and D. Pollack \cite{mpp}, R. Mazzeo and F. Pacard \cite{mp}, B. White \cite[\S 3]{Whi},  J. P\'erez and A. Ros \cite[Thm~6.7]{PerRos96}, and, finally in Reference~\cite{BetPicSic2013}. Let $\mathbf 0$ denote the zero function on $\Sigma$.
\begin{proposition}\label{thm:kapouleas}
Assume that the functions $f^{\mathbf 0,\Lambda_0}_i$, $i=1,2,3$, are linearly independent\footnote{The linear independence assumption is always satisfied when the genus of $\Sigma$ is greater than $1$, see Remark~\ref{thm:remLinIndip}.}.
Then, for sufficiently small neighborhoods $\mathfrak U_0$ and $\mathfrak V_0$:
\begin{equation}\label{eq:crucialeq}
\widetilde{\mathcal H}^{-1}(\mathbf 0)=\Big\{\big(\varphi,0,0,0,[\Lambda]\big):\ \text{$x_{\varphi,\Lambda}$ is $g_\Lambda$-minimal}\Big\}.
\end{equation}
\end{proposition}
\begin{proof}
First, we choose $\mathfrak U_0$ and $\mathfrak V_0$ small enough so that the functions $f^{\varphi,\Lambda}_i$, $i=1,2,3$, are linearly independent for all $(\varphi,\Lambda)\in\mathfrak U_0\times\mathfrak V_0$.
Now, let $(\varphi,a_1,a_2,_3,\Lambda)$ be such that:
\[\mathcal H^\Lambda(x_{\varphi,\Lambda})+\sum_{i=1}^3a_if^{\varphi,\Lambda}_i=0.\]
In order to prove \eqref{eq:crucialeq}, we need to show that from the above equality it follows
$a_1=a_2=a_3=0$.
Multiplying both sides of the above equality by $\sum_{i=1}^3a_if^{\varphi,\Lambda}_i$ we get:
\begin{equation}\label{eq:sumequalzero}
\mathcal H^\Lambda(x_{\varphi,\Lambda})\sum_{i=1}^3a_if^{\varphi,\Lambda}_i+\left(\sum_{i=1}^3a_if^{\varphi,\Lambda}_i\right)^2=0.
\end{equation}
We claim that for all $i=1,2,3$ we have:
\begin{equation}\label{eq:termequalzero}
\int_\Sigma\mathcal H^\Lambda(x_{\varphi,\Lambda})f^{\varphi,\Lambda}_i\,\mathrm d\Sigma_{\varphi,\Lambda}=0,
\end{equation}
where $\mathrm d\Sigma_{\varphi,\Lambda}$ is the volume element of the pull-back by $x_{\varphi,\Lambda}$
of $g_\Lambda$. This follows from Stokes' Theorem, observing that:
\begin{equation}\label{eq:compdivergence}
\mathcal H^\Lambda(x_{\varphi,\Lambda})f^{\varphi,\Lambda}_i=\mathrm{div}(K_i^{\varphi,\Lambda}),
\end{equation}
where the  $K_i^{\varphi,\Lambda}$'s  are defined in \eqref{eq:Kilambdaphi}, see Lemma~\ref{thm:divergence}.
Using \eqref{eq:sumequalzero} and \eqref{eq:termequalzero} we get:
\[\int_\Sigma\big(\sum_{i=1}^3a_if^{\varphi,\Lambda}_i\big)^2\,\mathrm d\Sigma_{\varphi,\Lambda}=0,
\]
which gives $\sum_{i=1}^3a_if^{\varphi,\Lambda}_i=0$. Since  the $f^{\varphi,\Lambda}_i$'s are
linearly independent, we obtain $a_1=a_2=a_3=0$, which proves our result.
\end{proof}
\section{Local rigidity}
\label{sec:locrigidity}
Fix $[\Lambda]\in\Flat(\mathds T^3)$; two embeddings $x_1,x_2\colon\Sigma\to\mathds R^3/\Lambda$
will be called \emph{$\Lambda$-congruent} if there exists a diffeomorphism $\psi\colon\Sigma\to\Sigma$ and
an element $\mathfrak t_0\in\mathds R^3/\Lambda$ such that $x_2=\mathfrak t_0+(x_1\circ\psi)$.
Observe that of $x_1$ and $x_2$ are $\Lambda$-congruent, and $x_1$ is $g_\Lambda$-minimal,
then also $x_2$ is $g_\Lambda$-minimal.\smallskip

Let $\mathfrak V$ be a neighborhood of $[\Lambda_0]$ in $\Flat(\mathds T^3)$, and let
$\mathfrak V\ni[\Lambda]\mapsto\varphi_\Lambda\in\mathfrak U_0$ be a continuous map such that
$x_{\varphi_\Lambda,\Lambda}$ is $g_\Lambda$-minimal for all $[\Lambda]\in\mathfrak V$.
\begin{definition}
The family of minimal embeddings $(x_{\varphi_\Lambda,\Lambda})_{[\Lambda]\in\mathfrak V}$ is said
to be \emph{locally rigid at $\Lambda_0$} if for every $[\Lambda]\in\mathfrak V$ and every $g_\Lambda$-minimal
embedding $y\colon\Sigma\to\mathds T^3$ sufficiently close to $x_0$,
$y$ is $\Lambda$-congruent to $x_{\varphi_\Lambda,\Lambda}$.
\end{definition}
In order to formulate a rigidity criterion, let us introduce a suitable notion of nondegeneracy for minimal
embeddings in $(\mathds T^3,g_\Lambda)$.
\begin{definition}\label{thm:defequivnondegeneracy}
Assume that $x_{\varphi,\Lambda}$ is a $g_\Lambda$-minimal embedding, and let
$J_{\varphi,\Lambda}\colon C^{2,\alpha}(\Sigma)\to C^{0,\alpha}(\Sigma)$ denote its Jacobi\footnote{%
$J_{\varphi,\Lambda}$ is the elliptic operator on $\Sigma$ given by $\Delta_{\varphi,\Lambda}-\Vert\mathcal S_{\varphi,\Lambda}\Vert^2$, where $\Delta_{\varphi,\Lambda}$ is the (positive) Laplacian of the pull-back of the metric $g_\Lambda$ by $x_{\varphi,\Lambda}$,  $\mathcal S_{\varphi,\Lambda}$ is the second fundamental form of $x_{\varphi,\Lambda}$, and $\Vert\cdot\Vert$ is the Hilbert--Schmidt norm.}
operator.
We say that $x_{\varphi,\Lambda}$ is \emph{equivariantly nondegenerate} if $\mathrm{Ker}(J_{\varphi,\Lambda})=\mathrm{span}\big\{f_1^{\varphi,\Lambda},f_2^{\varphi,\Lambda},f_3^{\varphi,\Lambda}\big\}$.
\end{definition}
The span of $\big\{f_1^{\varphi,\Lambda},f_2^{\varphi,\Lambda},f_3^{\varphi,\Lambda}\big\}$ is the space of the so-called \emph{Killing--Jacobi} fields along $x_{\varphi,\Lambda}$. Thus, an equivalent way of characterizing
equivariant nondegeneracy is the fact that every Jacobi field along $x_{\varphi,\Lambda}$ is a Killing--Jacobi field.

A direct application of the equivariant implicit function theorem proved in \cite{BetPicSic2013}
gives the following:
\begin{theorem}\label{thm:IFT}
Assume that the functions $f^{\mathbf 0,\Lambda_0}_i$, $i=1,2,3$, are linearly independent\footnote{\label{foo:1}A statement similar to that of Theorem~\ref{thm:IFT} holds without the linear independence assumption, with suitable modifications of the
function $\widetilde{\mathcal H}$ in \eqref{eq:centralmap}. However, we observe that such assumption is always satisfied when the genus of $\Sigma$ is greater than $1$, see Remark~\ref{thm:remLinIndip}.},
and that $x_0$ is equivariantly nondegenerate. Then, there exists a smooth function $\mathfrak V\ni\Lambda\mapsto\varphi_\Lambda\in\mathfrak U_0$ defined in a neighborhood $\mathfrak U$ of $[\Lambda_0]$
in $\Flat(\mathds T^3)$, such that:
\begin{itemize}
\item[(a)] $\Lambda_\mathbf 0=\Lambda_0$;
\item[(b)] $x_{\varphi_\Lambda,\Lambda}$ is a minimal $g_\Lambda$ embedding for all $[\Lambda]\in\mathfrak U$;
\item[(c)] the family $(x_{\varphi_\Lambda,\Lambda})_{[\Lambda]\in\mathfrak V}$ is
locally rigid at $\Lambda_0$.
\end{itemize}
Therefore, near $x_0$, all triply periodic minimal surfaces in ${\mathds R}^3$ consist of six-parameter family of surfaces.
If we restrict ourselves to the unit volume lattice, then, near $x_0$, all triply periodic minimal surfaces consist of a five-parameter family of pairwise non-homothetic surfaces.
\end{theorem}

\begin{proof}
(a) - (c) follows from \cite[Theorem~5.2]{{BetPicSic2013}}.
Then, we know the dimension of the space of triply periodic minimal surfaces near $x_0$ from the dimension of the isometry class in the flat $3$-torus.
\end{proof}

\section{Linearization}
In order to study the lack of local rigidity for a family of minimal $g_\Lambda$-embeddings,
we study the linearization of the map $\widetilde{\mathcal H}_\Lambda$, given in \eqref{eq:centralmap2},
at one of its zeros, described in Proposition~\ref{thm:kapouleas}. Let $(\varphi,\Lambda)\in\mathfrak U_0\times\mathfrak V_0$ be such that
$\widetilde{\mathcal H}_\Lambda(\varphi,0,0,0)=0$.
Let us denote by:
\begin{equation}\label{eq:linearmapT}
T_{\varphi,\Lambda}:C^{2,\alpha}(\Sigma)\times\mathds R^3\longrightarrow C^{0,\alpha}(\Sigma)
\end{equation}
the bounded linear operator:
\[T_{\varphi,\Lambda}=\mathrm d\widetilde{\mathcal H}_\Lambda(\varphi,0,0,0).\]
\begin{proposition}\label{eq:FredholmnessT}
The operator $T_{\varphi,\Lambda}$ is given by:
\begin{equation}\label{eq:formofderivative}
T_{\varphi,\Lambda}(\psi,b_1,b_2,b_3)=J_{x_{\varphi,\Lambda}}(\psi)+\sum_{i=1}^3b_if_i^{\varphi,\Lambda},
\end{equation}
for all $(\psi,b_1,b_2,b_3)\in C^{2,\alpha}(\Sigma)\times\mathds R^3$,
where $J_{x_{\varphi,\Lambda}}$ is the Jacobi operator of the $g_\Lambda$-minimal embedding
$x_{\varphi,\Lambda}$. This is a Fredholm operator of index equal to $3$.
If the $f^{\varphi,\Lambda}_i$'s are linearly independent, then $T_{\varphi,\Lambda}$ is surjective if and only if $x_{\varphi,\Lambda}$ is an equivariantly
nondegenerate $g_\Lambda$-minimal embedding.
\end{proposition}
\begin{proof}
It is well known that the differential of the mean curvature map
$\varphi\mapsto\mathcal H(x_{\varphi,\Lambda})$  at a minimal embedding is given by the Jacobi operator
$J_{x_{\varphi,\Lambda}}$. Equality \eqref{eq:formofderivative} follows easily, observing that:
\begin{itemize}
\item the map $\mathds R^3\ni (a_1,a_2,a_3)\longmapsto\sum_{i=1}^3a_if_i^{\varphi,\Lambda}\in C^{0,\alpha}(\Sigma)$ is linear;
\item the differential of the map $\varphi\mapsto f_i^{\varphi,\Lambda}$ is not involved in formula
\eqref{eq:formofderivative}, since $\mathrm d\widetilde{\mathcal H}_\Lambda$ is computed at $a_1=a_2=a_3=0$.
\end{itemize}
As to the Fredholmness, it is well known that $J_{x_{\varphi,\Lambda}}:C^{2,\alpha}(\Sigma)\to C^{0,\alpha}(\Sigma)$ is Fredholm, and it has index $0$ (it is an elliptic differential operator), and so the operator
$C^{2,\alpha}(\Sigma)\times\mathds R^3\ni(\psi,b_1,b_2,b_3)\mapsto J_{\varphi,\Lambda}(\psi)\in C^{0,\alpha}(\Sigma)$ is Fredholm of index $3$. Clearly, $T_{\varphi,\Lambda}$ is a finite rank perturbation of such operator, and therefore it is also a Fredholm operator of index $3$.

As to the last statement, note that $J_{\varphi,\Lambda}$ is symmetric with respect to the
$L^2$-pairing (using the volume element of $g_\Lambda$), and that its image is the $L^2$-orthogonal of its (finite dimensional) kernel. Such kernel contains the span of the $f^{\varphi,\Lambda}_i$'s, and it is equal
to this span when $x_{\varphi,\Lambda}$ is equivariantly nondegenerate.
Clearly:
\[\mathrm{Range}(T_{\varphi,\Lambda})=\mathrm{Range}(J_{\varphi,\Lambda})+\mathrm{span}\big\{f_1^{\varphi,\Lambda},f_2^{\varphi,\Lambda},f_3^{\varphi,\Lambda}\big\},\]
and the conclusion follows easily.
\end{proof}
\section{Bifurcation}
Let us now assume that $\left[-\varepsilon,\varepsilon\right]\ni s\mapsto\big(\varphi_s,\Lambda_{(s)}\big)\in\mathfrak U_0\times\mathfrak V_0$ is a continuous map such that:
\begin{itemize}
\item[(i)] $x_{\varphi_s,\Lambda_{(s)}}$ is a $g_{\Lambda_{(s)}}$-minimal embedding for all $s$;
\item[(ii)] $\varphi_0=\mathbf 0$ and $\Lambda_{(0)}=\Lambda_0$.
\end{itemize}
\begin{definition}
We say that $s=0$ is a bifurcation instant for the path $s\mapsto x_{\varphi_s,\Lambda_{(s)}}$
if there exists a sequence $(s_n)_{n\in\mathds N}\subset\left]-\varepsilon,\varepsilon\right[$  and a sequence $x_n\colon\Sigma\to\mathds T^3$ of embeddings such that:
\begin{itemize}
\item $\lim\limits_{n\to\infty}s_n=0$ and $\lim\limits_{n\to\infty}x_n=x_0$ (in the $C^{2,\alpha}$-topology);
\item $x_n$ is $g_{\Lambda_{(s_n)}}$-minimal for all $n$;
\item $x_n$ is \emph{not} $\Lambda_{(s_n)}$-congruent to $x_{\varphi_{s_n},\Lambda_{(s_n)}}$ for all $n$.
\end{itemize}
\end{definition}
In particular, if $s=0$ is a bifurcation instant, then no family $(x_{\varphi_\Lambda,\Lambda})_{[\Lambda]\in\mathfrak V}$ that contains the path $s\mapsto\big(\varphi_s,\Lambda_{(s)}\big)$ is locally rigid at $\Lambda_0$. Thus, by Theorem~\ref{thm:IFT}, bifurcation can occur at $s=0$ only if $x_0$ is
a $\Lambda_0$-minimal equivariantly degenerate embedding.
In the situation above, the sequence $x_n$ possibly belongs to a continuous set of minimal embeddings having larger cardinality, which is usually called the \emph{bifurcating branch} issuing from $x_{\varphi_0,\Lambda_{(0)}}$,
while the family $s\mapsto x_{\varphi_s,\Lambda_{(s)}}$ is called the \emph{trivial branch}.
\begin{remark}\label{thm:remgenuinebifurcation}
Let us observe that in the definition of bifurcation given above, it is not required that, near $s=0$, the metrics
$g_{\Lambda_{(s)}}$ should be pairwise non homothetic. Under this additional hypothesis, a stronger conclusion about the bifurcation branch can be drawn. Namely, if the flat metrics $g_{\Lambda(s)}$ are pairwise non homothetic near $s=0$, then every embedding $x_n$ in the bifurcating branch is not homothetic to any of the minimal surfaces in the trivial branch.
\end{remark}
The notion of Morse index is central in Bifurcation Theory. Let $x_{\varphi,\Lambda}$ be
a $g_\Lambda$-minimal embedding.
\begin{definition}\label{thm:defMorseindex}
The \emph{Morse index} $\mathrm i_\text{Morse}(\varphi,\Lambda)$ of $x_{\varphi,\Lambda}$
is the number of negative eigenvalues of the Jacobi operator $J_{\varphi,\Lambda}$, counted with
multiplicity.
\end{definition}
The number $\mathrm i_\text{Morse}(\varphi,\Lambda)$ is in fact the Morse index of $x_{\varphi,\Lambda}$
as a critical point of the $g_\Lambda$-area functional defined in the space of embeddings of $\Sigma$ into
$\mathds T^3$. A sufficient condition for variational bifurcation is given in terms of jumps of the Morse index.
Here we cannot employ directly variational techniques, in that our equation $\widetilde{\mathcal H}=0$ is not variational,
and we have to resort to a weaker bifurcation result for general Fredholm operators.
This requires a certain parity change in the negative spectrum of the path of operators, which corresponds to
an \emph{odd} jump of the Morse index.
\begin{theorem}\label{thm:mainbiftheorem}
Let $\left[-\varepsilon,\varepsilon\right]\ni s\mapsto\big(\varphi_s,\Lambda_{(s)}\big)\in\mathfrak U_0\times\mathfrak V_0$ be a continuous map satisfying (i) and (ii) above. Assume the following:
\begin{itemize}
\item[(a)] $x_{\varphi_s,\Lambda_{(s)}}$ is equivariantly nondegenerate for all $s\ne0$;
\item[(b)] the functions $f^{\mathbf 0,\Lambda_0}_i$, $i=1,2,3$, are linearly independent;
\item[(c)] $\mathrm i_\text{Morse}(\varphi_{-\varepsilon},\Lambda_{(-\varepsilon)})-
\mathrm i_\text{Morse}(\varphi_{\varepsilon},\Lambda_{(\varepsilon)})$ is an odd integer.
\end{itemize}
Then, $s=0$ is a bifurcation instant for the path $s\mapsto\big(\varphi_s,\Lambda_{(s)}\big)$.
\end{theorem}
\begin{proof}
A precise statement of the bifurcation theorem employed in this proof is given in Appendix~\ref{app:technical},
Section~\ref{sub:bifresult}.
For the reader's convenience, we will refer to the assumptions of this result throughout the proof.

By (a), the integer valued function $s\mapsto \mathrm i_\text{Morse}(\varphi_{s},\Lambda_{-s})$
is constant on $\left[-\varepsilon,0\right[$ and on $\left]0,\varepsilon\right]$.
Thus, we can choose arbitrarily small values of $\varepsilon$ and reduce the size of $\mathfrak U_0$ and $\mathfrak V_0$ when needed, maintaining the validity of assumption
(c).

First, by continuity, we can assume that the functions $f^{\varphi,\Lambda}_i$, $i=1,2,3$, are linearly independent for all fixed $\big(\varphi,[\Lambda]\big)\in\mathfrak U_0\times\mathfrak V_0$.
Second, we choose a codimension $3$ closed subspace $X$ of $C^{2,\alpha}(\Sigma)$ which is
transversal to $Y_0\colon=\mathrm{span}\big\{f^{\mathbf 0,\Lambda_0}_1,f^{\mathbf 0,\Lambda_0}_2,f^{\mathbf 0,\Lambda_0}_3\big\}$;
for instance, $X$ can be taken to be the $L^2$-orthogonal of $Y_0$ relatively to the volume element of $g_{\Lambda_0}$.
Again, by continuity, we can assume that $X$ is transversal\footnote{%
Note that transversality, i.e., $X+Y_{\varphi,\Lambda}=C^{2,\alpha}(\Sigma)$,  also implies $X\cap Y_{\varphi,\Lambda}=\{0\}$,
by a dimension argument.} to
$Y_{\varphi,\Lambda}:=\mathrm{span}\big\{f^{\varphi,\Lambda}_1,f^{\varphi,\Lambda}_2,f^{\varphi,\Lambda}_3\big\}$, for all $\big(\varphi,[\Lambda]\big)\in\mathfrak U_0\times\mathfrak V_0$.

For all $\Lambda$, the group $G_\Lambda=\mathds R^3/\Lambda$ acts isometrically on $(\mathds T^3,g_{\Lambda_s})$
by translation, and this defines a smooth action on the set of embeddings of $\Sigma$ into $\mathds T^3$.
Passing to the quotient by the action of the diffeomorphism group of $\Sigma$, we have a continuous
action on the set of unparameterized embeddings, and therefore a local action\footnote{More precisely, the
definition of the local action of $G_\Lambda$ on $\mathfrak U_0$ is as follows. For $\mathfrak t\in\mathds R^3/\Lambda$ close to $0$, and $\varphi\in\mathfrak U_0$, consider the embedding $y=\mathfrak t+x_{\varphi,\Lambda}$. There exists a unique $\varphi'\in\mathfrak U_0$ such that $x_{\varphi',\Lambda}$ is
a reparameterization of $y$. Then, $\mathfrak t\cdot\varphi=\varphi'$.} on the open set  $\mathfrak U_0$.
 The $G_\Lambda$-orbit of every smooth embedding, and in particular, of any minimal embedding, is a smooth
 submanifold of $\mathfrak U_0$; details of the proof of this fact can be found in \cite{AliPic10}.
 Note that the $G_\Lambda$-orbit of an (unparameterized) embedding $x$ is precisely the set of
 (unparameterized) embeddings that are $\Lambda$-congruent to $x$.
 Given $[\Lambda]\in\mathfrak V_0$ and any smooth function $\varphi\in\mathfrak U_0$, the
 tangent space at $\varphi$ of the $G_\Lambda$-orbit of $\varphi$ is exactly the $3$-dimensional
 space $Y_{\varphi,\Lambda}$. By transversality, if $\mathfrak U_0$ is small enough, we can assume
 that for every smooth $\varphi\in\mathfrak U_0$, there is a unique intersection point $\varphi^\Lambda$ between the orbit  $G_\Lambda\cdot\varphi$ and $X\cap\mathfrak U_0$.
 The path $s\mapsto\varphi_s^{\Lambda_{(s)}}$ is continuous,
 and up to replacing $\varphi_s$ with
 $\varphi_s^{\Lambda_{(s)}}$, we can therefore assume that $\varphi_s\in X\cap\mathfrak U_0$ for all $s$.
 This settles  assumption (B)  in Section~\ref{sub:bifresult}.

 Finally, there is a correspondence between zeros of the function $\widetilde{\mathcal H}$ in $\mathfrak U_0\times\mathds R^3\times\mathfrak V_0$, defined in \eqref{eq:centralmap}, and its restriction to $(X\cap\mathfrak U_0)\times\mathds R^3\times\mathfrak V_0$: if $\big(\varphi,[\Lambda]\big)\in\mathfrak U_0\times\mathfrak V_0$ is such that $\widetilde{\mathcal H}(\varphi,0,0,0,\Lambda)=0$, i.e., $x_{\varphi,\Lambda}$ is a (smooth) $g_\Lambda$-minimal embedding,
 then also $\widetilde{\mathcal H}(\varphi^\Lambda,0,0,0,\Lambda)=0$.\smallskip

In conclusion, the argument above shows  that $\Lambda$-congruence classes of $g_\Lambda$-minimal embeddings of $\Sigma$ into $\mathds T^3$ are into 1-1 correspondence with zeros of the function $\widetilde{\mathcal H}$ in  $(X\cap\mathfrak U_0)\times\mathds R^3\times\mathfrak V_0$.
The aimed bifurcation result will then be proved in this context, and it will be obtained as a direct
application of a classical bifurcation theorem for Fredholm operators, see \cite[Theorem~II.4.4, p.\ 212]{Kiel2012}.
A precise statement of this theorem is recalled in Appendix~\ref{app:technical}.\smallskip

 If $s\in\left[-\varepsilon,\varepsilon\right]$ is such that $\widetilde{\mathcal H}_{\Lambda_s}(\varphi_s,0,0,0)=0$, then the restriction of $T_{s}=\mathrm d\widetilde{\mathcal H}_{\Lambda_s}(\varphi_s,0,0,0)$
 to $X\times\mathds R^3$, denoted by $\overline T_s$ see \eqref{eq:linearmapT}, is  a Fredholm operator of index $0$,
 which settles assumption (C1) in Section~\ref{sub:bifresult}.
 This follows easily from Proposition~\ref{eq:FredholmnessT}, since $X$ is transversal to $Y_{\varphi_s,\Lambda_{(s)}}$, which is a $3$-dimensional
 subspace of $\mathrm{Ker}(T_s)$. By the same argument, Proposition~\ref{eq:FredholmnessT}
 says that $x_{\varphi_s,\Lambda_s}$ is an equivariantly nondegenerate
 $g_{\Lambda_{(s)}}$-minimal embedding if and only if $\overline T_s$  is
 an isomorphism. Thus, assumption (a) implies that, for $s\ne0$, $\overline T_s$ is nonsingular,
 which settles assumption (BT1) in Section~\ref{sub:bifresult}.\smallskip

Let us the injective continuous linear map%
\begin{equation}\label{eq:identification}
X\oplus\mathds R^3\ni(\psi,b_1,b_2,b_3)\longmapsto\psi+\sum_{i=1}^3b_if_i^{\mathbf 0,\Lambda_0}\in C^{0,\alpha}(\Sigma)
\end{equation}%
to identify $X\oplus\mathds R^3$ with a subspace\footnote{Using the identification \eqref{eq:identification},
the operator $\overline T_0\colon C^{2,\alpha}(\Sigma)\to\ C^{0,\alpha}(\Sigma)$ is given by
$J_{\mathbf 0,\Lambda_0}+P_{0}$, where $P_{0}\colon X\oplus Y_0\to Y_0$ is the projection.} of $C^{0,\alpha}(\Sigma)$,
as in assumption (A) in Section~\ref{sub:bifresult}.
Since $\overline T_0$ is Fredholm,  then $0$ is an isolated eigenvalue of $\overline T_0$ (assumption (D) in Section~\ref{sub:bifresult}), and it
has finite multiplicity, given by  $m=\mathrm{dim}\big(\mathrm{Ker}(T_0)\big)-3>0$.
Notice that $\overline T_0$ is diagonalizable: it coincides with the Jacobi operator $J_{\mathbf 0,\Lambda_0}$ on $X$,
which is $J_{\mathbf 0,\Lambda_0}$-invariant, and it is the identity on $\mathds R^3$. In particular,
the generalized $0$-eigenspace of $T_0$, i.e., $E_0\colon=\bigcup_{k\ge1}\mathrm{Ker}(\overline T_0^k)$, coincides
with the kernel of $\overline T_0$, given by $\mathrm{Ker}(J_{\mathbf 0,\Lambda_0})\cap X$.
\smallskip

Using the identification \eqref{eq:identification}, the operators $\overline T_s$ can be seen as unbounded linear operators
on $C^{0,\alpha}(\Sigma)$, with domain $C^{2,\alpha}(\Sigma)$. As such, they are \emph{closed} operators, i.e., they have
closed graphs. This follows easily observing that they are finite rank perturbations\footnote{The sum of a closed and a bounded
operator is closed.} of the self-adjoint elliptic operators of
second order $J_{\varphi_s,\Lambda_{(s)}}$, which are closed (see for instance \cite[Section~III.1]{Kiel2012}).
This settles assumption (C2) in Section~\ref{sub:bifresult}.
\smallskip

Let us now show that the path of Fredholm operators $\overline T_s$ has an odd crossing number at $s=0$
using assumption (c).

To this aim, let us consider the continuous path of Fredholm operators $\overline T'_s=J_{\varphi_s,\Lambda_{(s)}}+P_s$,
where $P_s:C^{2,\alpha}(\Sigma)\cong X_s\oplus Y_s\to Y_s$ is the projection, and $X_s$ is the $L^2$-orthogonal
complement of $Y_s$ relatively to the metric $g_{\Lambda_{(s)}}$. Observe that:
\begin{enumerate}
\item[(i)] $\overline T'_0=\overline T_0$;
\item[(ii)] $\overline T'_s$ is invertible for all $s\ne0$;
\item[(iii)] $\overline T'_s$ is diagonalizable with real eigenvalues, and for $s\ne0$, its spectrum
$\mathrm{spec}(\overline T_s')$ coincides with $\big(\mathrm{spec}(J_{\varphi_s,\Lambda_{(s)}})\setminus\{0\}\big)\bigcup\{1\}$;
\item[(iv)] $\overline T_s'$ has an odd crossing number at $s=0$, by assumption (c).
\end{enumerate}
Statement (iv) follows easily from (iii).
In order to conclude that also the family $\overline T_s$ has an odd crossing number at $s=0$, it suffices
to show that there exists $r$ arbitrarily small and $s=s(r)$ such that the isomorphisms $\overline T_s$ and $\overline T'_s$ are
endpoints of a continuous path of invertible operators that remain inside the ball $B(\overline T_0,r)$ of radius $r$ centered
at $\overline T_0$, see Remark~\ref{thm:simplifyingobs}.

The difference $\overline T'_s-\overline T_s$ is equal to:
\begin{itemize}
\item $P_s$ on $X_0$;
\item $(J_{\varphi_s,\Lambda_{(s)}}-J_{\mathbf 0,\Lambda_0})+(P_s-\mathrm{I}_s)$ on $Y_0$,
\end{itemize}
where $I_s\colon Y_0\to Y_s$ is the isomorphism defined by $\mathrm I_s\big(f_i^{\mathbf 0,\Lambda_{(0)}}\big)=f_i^{\varphi_s,\Lambda_{(s)}}$,
$i=1,2,3$. Since $\lim\limits_{s\to0}P_s=P_0$ and $\lim\limits_{s\to0}J_{\varphi_s,\Lambda_{(s)}}=J_{\mathbf 0,\Lambda_0}$ in the operator norm, then:
\[\lim_{s\to0}\Big\Vert P_s\big\vert_{X_0}\Big\Vert=0,\quad\text{and}\quad\lim_{s\to0}\Big\Vert\big[(J_{\varphi_s,\Lambda_{(s)}}-J_{\mathbf 0,\Lambda_0})+(P_s-\mathrm{I}_s)\big]\big\vert_{Y_0}\Big\Vert=0.\]
Observe that $P_s$ is an operator of rank $3$, and $Y_0$ has dimension $3$.
In other words, $\overline T'_s$ is the sum $\overline T_s+R_s$, with $R_s$ a finite rank operator such that
$\lim\limits_{s\to0}\Vert R_s\Vert=0$. This implies easily that, given any $r>0$, there exists $s=s(r)$ such
that both $\overline T_s$ and $\overline T_s'$ belong to the ball $B(\overline T_0,r)$, and
that there exists a continuous path of invertible operators in $B(\overline T_0,r)$ joining $\overline T_s$ and
$\overline T_s'$. This shows that the family $\overline T_s$ has an odd crossing number at $s=0$ (assumption (BT2) in Section~\ref{sub:bifresult}), and concludes the proof.
\end{proof}
\begin{remark}\label{thm:remLinIndip}
Assumption~(b) in Theorem~\ref{thm:mainbiftheorem} is always satisfied when the genus of $\Sigma$ is greater than $1$.
There is a number of ways to prove this fact, here we propose the most elementary one. Recalling the definition of the
Jacobi field $f_i^{\mathbf 0,\Lambda_0}$ in Section~\ref{sub:notations}, observe that they are linearly independent if and only
if  some non zero constant (i.e., translation invariant) vector field of $\mathds R^3/\Lambda_0$ is everywhere tangent to
$S=x_0(\Sigma)$.
\begin{lemma}
Let $S$ be an embedded submanifold of $\mathds T^3$ which is diffeomorphic to a closed orientable surface of genus $\mathrm{gen}(S)>1$. Then, no nontrivial constant vector field of $\mathds T^3$ is everywhere tangent to $S$.
\end{lemma}
\begin{proof}
Nontrivial constant vector fields are never vanishing.  But the Euler characteristic
of $S$ is $2-2\,\mathrm{gen}(S)<0$, so there are no never vanishing vector fields on $S$ by the
Poincar\'e--Hopf theorem.
\end{proof}
More on the geometry of minimal submanifolds in flat tori can be found in reference \cite{NagSmy75}.
\end{remark}
\section{Proof of Theorems A, B, and C}
\label{sec:applications}
\begin{figure}
\includegraphics[scale=.35]{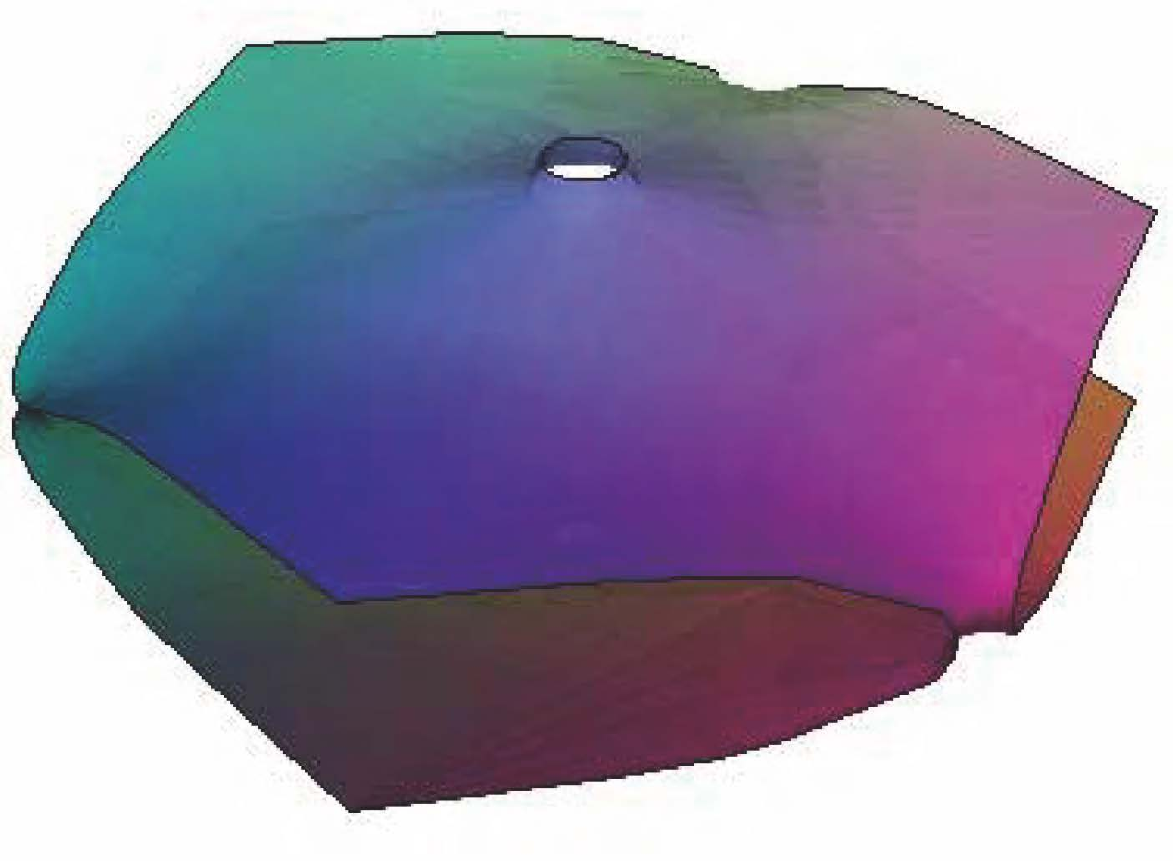}\includegraphics[scale=0.35]{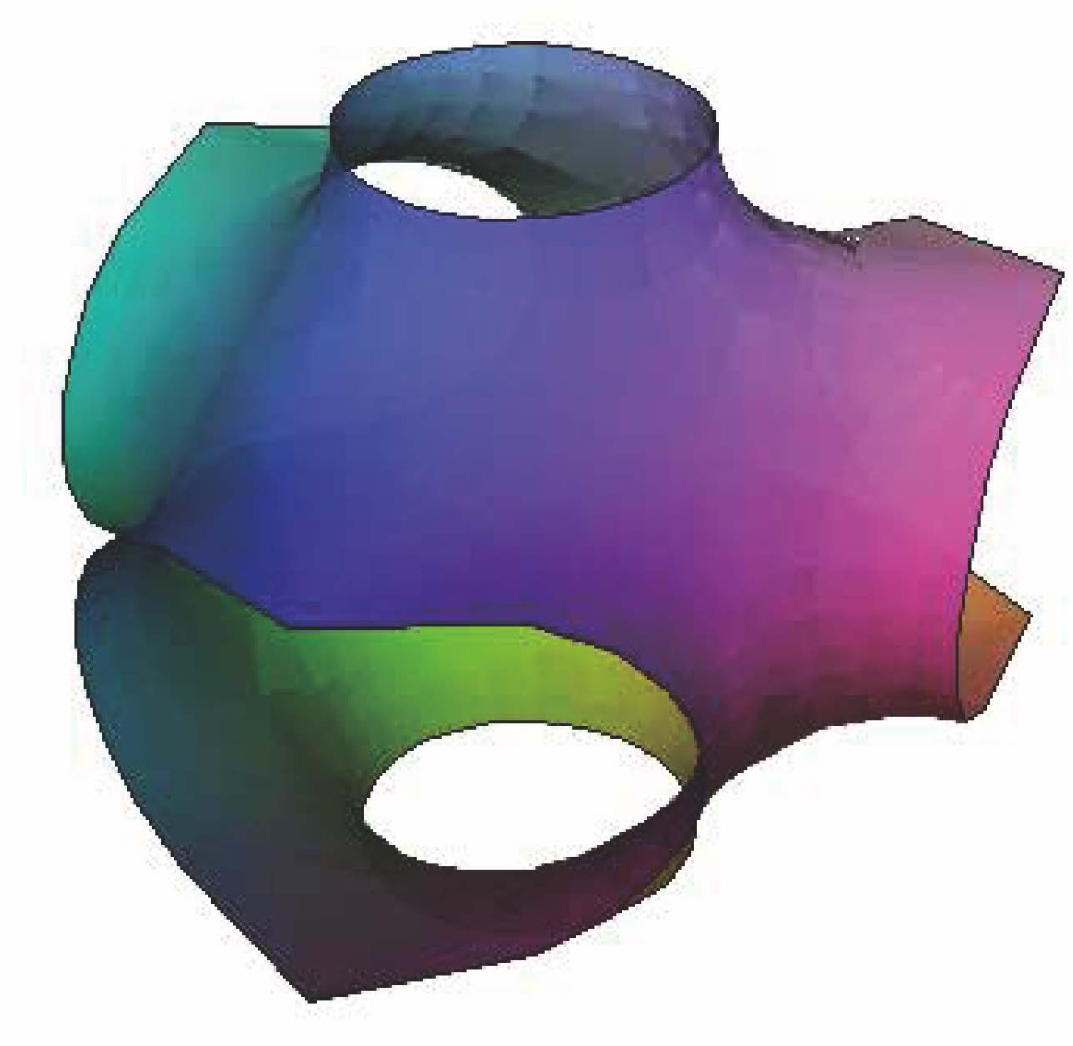}
\includegraphics[scale=0.35]{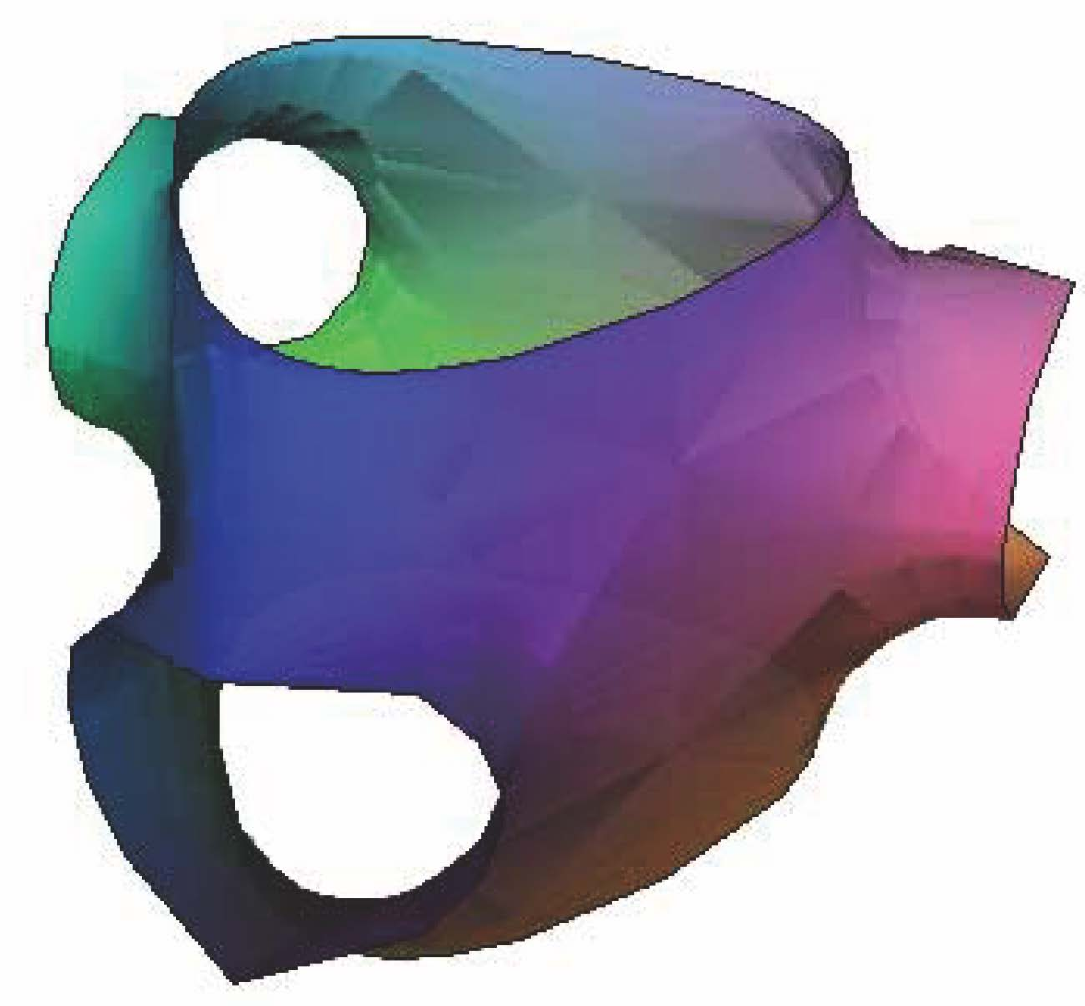}
\caption{Surfaces $M_a$ of the H-family, with $a=0.1$ (Morse index $2$), $a=0.5$ (Morse index $1$),
and $a=0.9$ (Morse index $3$). }
\label{fig:Hfamily}
\end{figure}
We can apply Theorem~\ref{thm:mainbiftheorem} and the results of \cite{E-S} to conclude the proof of Theorem~A.
\begin{proof}[Proof of Theorem~A]
For the H-family $(M_a)_{a\in\left]0,1\right[}$, see Example~\ref{exa:H}
and \cite[Main Theorem~1]{E-S}, the nullity of every
surface except for $a_0\approx 0.49701$ and $a_1\approx 0.71479$ is equal to $3$.
The Morse index of $M_a$ is equal to $2$ for $a\in\left]0,a_0\right[$, it is equal to $1$ for $a\in\left]a_0,a_1\right[$, and it is equal to $3$ for $a\in\left]a_1,1\right[$.
Thus, there is an odd jump at $a_0$, which is a bifurcation instant. We cannot infer the existence
of bifurcation at $a_1$, where the Morse index has an even jump.

As to the rPD-family $(M_a)_{a\in\left]0,\infty\right[}$, see Example~\ref{exa:rPD} and
\cite[Main Theorem~2]{E-S}, there are two odd jumps of the Morse index at $a_1\approx 0.494722$,  $a_2 = (a_1)^{-1}\approx 2.02133$.
The nullity at every other instant is equal to $3$.

For the tP-family and the tD-family, see Example~\ref{exa:tPtD} and \cite[Main Theorem~3]{E-S},
there are two odd jumps of the Morse index:
$a_1, a_2 \in \left]2,+\infty\right[$
of the Morse index, where $a_1\approx 7.40284$ and $a_2\approx 28.7783$. The nullity at every other instant is equal to $3$.

Observe that assumption (b) of Theorem~\ref{thm:mainbiftheorem} is satisfied in all cases, because
all the minimal surfaces have genus equal to $3$, see Remark~\ref{thm:remLinIndip}.
\end{proof}
\begin{example}[H-family --- Figure~\ref{fig:Hfamily}]
\label{exa:H}
For $a\in\left]0,1\right[$, let $M_a$ be a hyperelliptic Riemann surface of genus $3$
defined by $w^2=z(z^3-a^3) \left( z^3-\frac{1}{a^3} \right)$ and $f$ a conformal
minimal immersion given by
\[f(p)=\Re \int^p_{p_0} i\big(1-z^2,\,i(1+z^2),\,2 z\big)^\mathrm{t}\, \dfrac{d z}{w}.\]
$f(M_a)$ is called H-family.
\end{example}
\begin{example}[rPD family, Karcher's TT surface, see ref.~ \cite{Karcher89} --- Figure~\ref{fig:rPdfamily}]
\label{exa:rPD}
For $a \in\left]0, \infty\right[$, let $M_a$ be a hyperelliptic Riemann surface of genus $3$
defined by $w^2=z(z^3-a^3) \left( z^3+\frac{1}{a^3} \right)$ and $f$ a conformal
minimal immersion given by
\[f(p)=\Re \int^p_{p_0} \big(1-z^2,\,i(1+z^2),\,2 z\big)^\mathrm{t}\, \dfrac{d z}{w}.\]
$f(M_a)$ is called rPD family or Karcher's TT surface.
$M_{\sqrt{2}}$ gives the so-called Schwarz Primitive surface (Schwarz P surface),
and $M_{1/\sqrt{2}}$ gives the so-called Schwarz Diamond surface (Schwarz D surface).
\begin{figure}
\includegraphics[scale=.3]{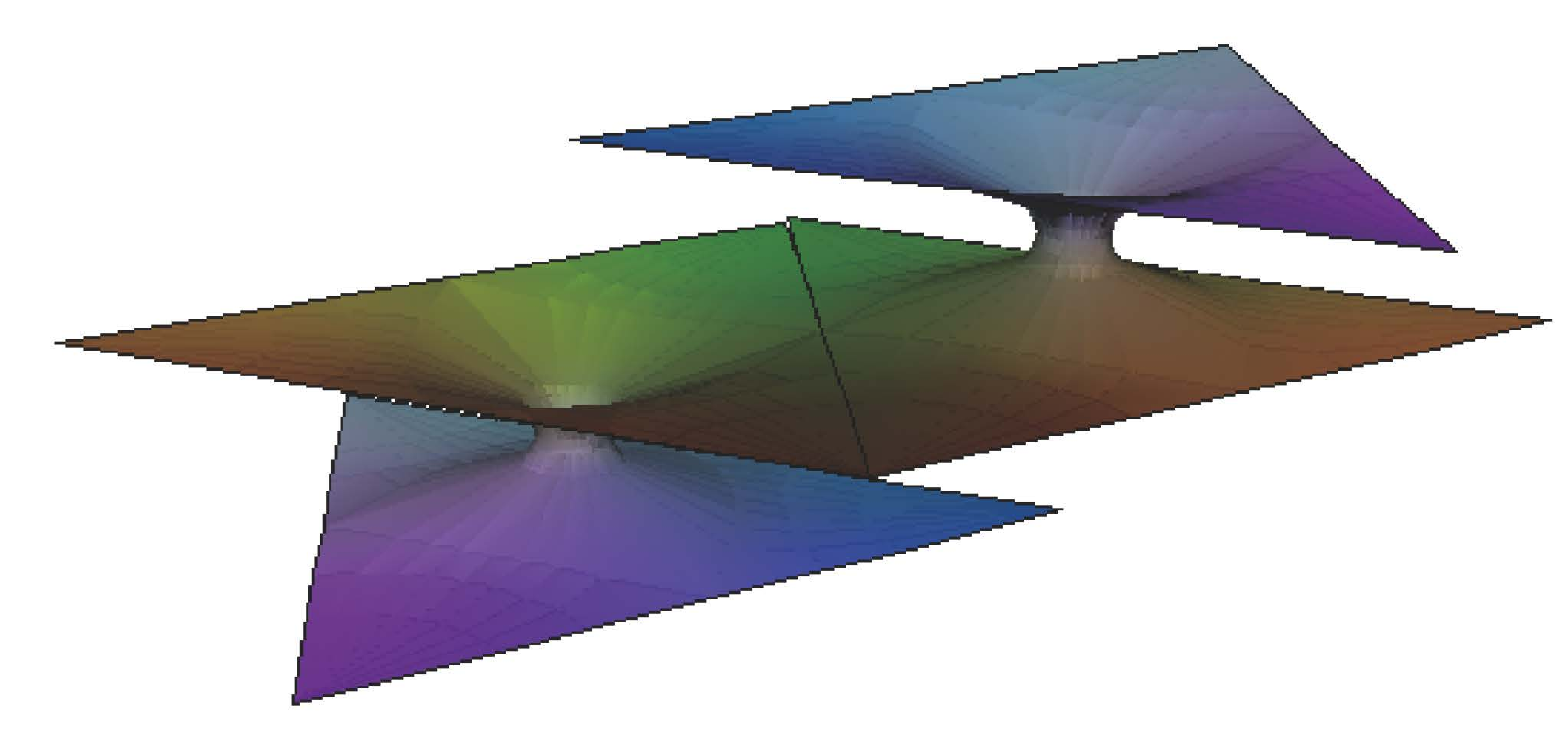}\quad\includegraphics[scale=.4]{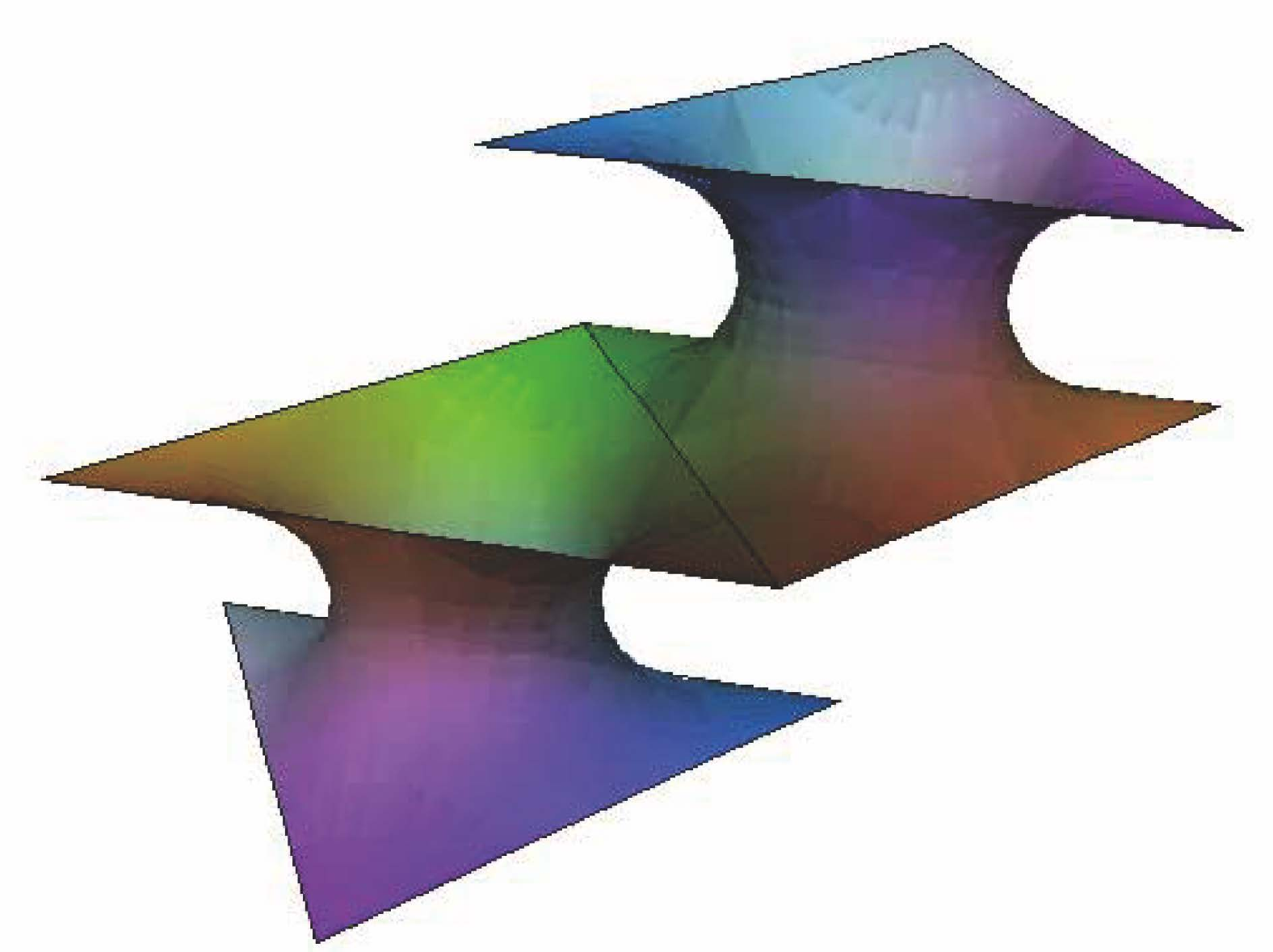}
\caption{Surfaces $M_a$ of the rPD-family, with $a=0.1$ (Morse index $2$), and $a=0.5$ (Morse index $1$).}
\label{fig:rPdfamily}
\end{figure}
\end{example}
\begin{example}[tP-family -- Figure~\ref{fig:tPfamily}, tD-family -- Figure~\ref{fig:tDfamily}]
\label{exa:tPtD}
For $a \in\left]2,+\infty\right[$, let $M_a$ be a hyperelliptic Riemann surface of genus $3$
defined by $w^2=z^8+a z^4+1$. Let $f$ be a conformal
minimal immersion given by
\[f(p)=\Re \int^p_{p_0} \big(1-z^2,\,i(1+z^2),\,2 z\big)^\mathrm{t}\, \dfrac{d z}{w}\]
and $f'$
\[f'(p)=\Re \int^p_{p_0} i\big(1-z^2,\,i(1+z^2),\,2 z\big)^\mathrm{t}\, \dfrac{d z}{w}.\]
$f(M_a)$ is called tP-family and $f'(M_a)$ is called tD-family.
$f(M_{14})$ gives the Schwarz P surface,
and $f'(M_{14})$ gives the Schwarz D surface.
\begin{figure}
\includegraphics[scale=.5]{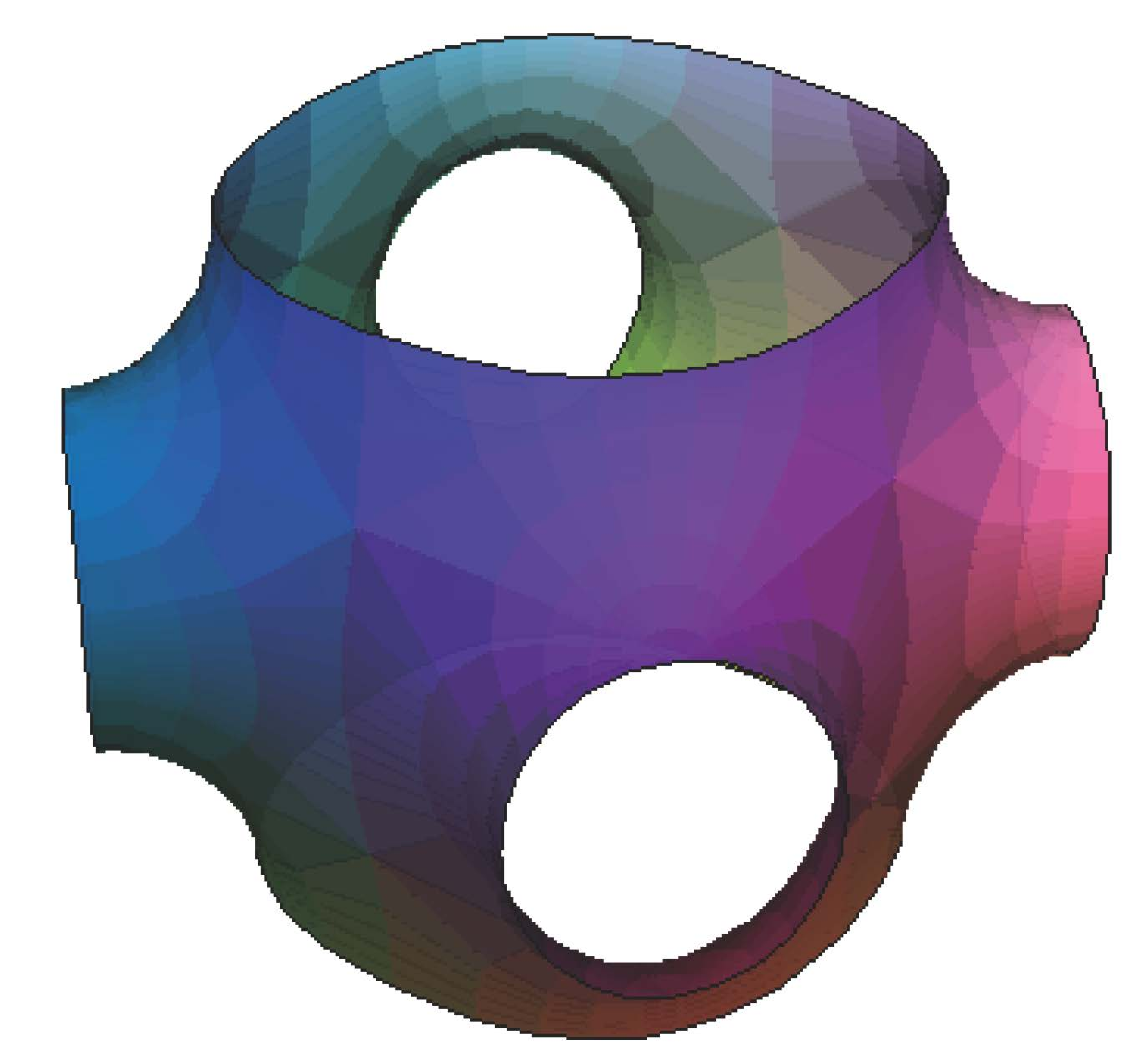}\quad\includegraphics[scale=.5]{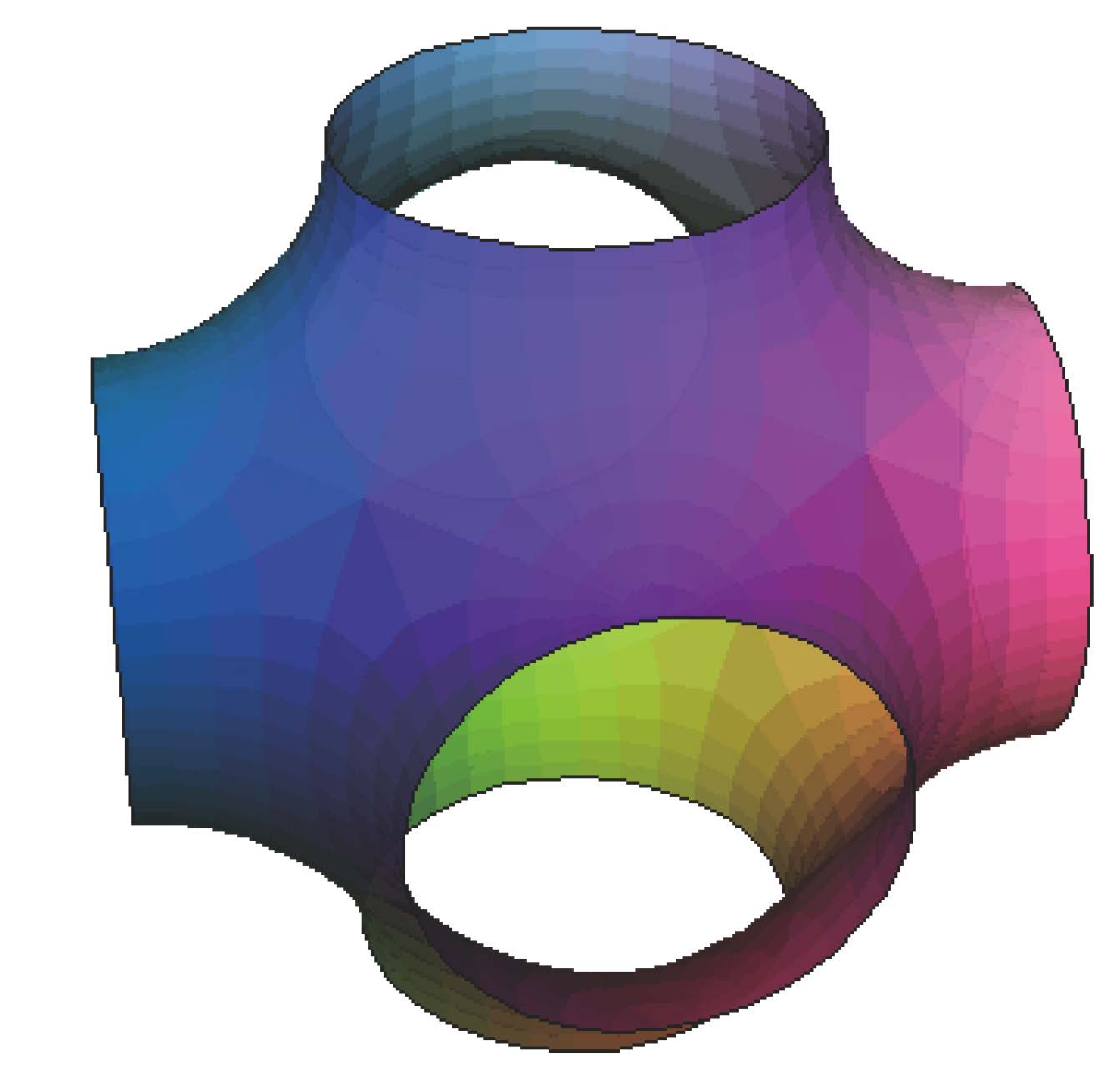}
\caption{Surfaces $M_a$ of the tP-family, with $a=0.285$ (Morse index $2$), and $a=14$ (Morse index $1$). The surface $M_{14}$ of the tP-family is also called \emph{Schwarz Primitive} surface, or \emph{P-surface}.}
\label{fig:tPfamily}
\end{figure}
\end{example}

As to the local rigidity, we can apply Theorem~\ref{thm:IFT} to the above families.
We have seen that there are exactly two surfaces that are not equivariantly nondegenerate in each of the rPD-family, the H-family, the tP-family and the tD-family.
There is a fifth family of triply periodic minimal surfaces, called the  tCLP-family
(see Example~\ref{exa:tCLP}), which consists of equivariantly nondegenerate surfaces.
\begin{figure}
\includegraphics[scale=.4]{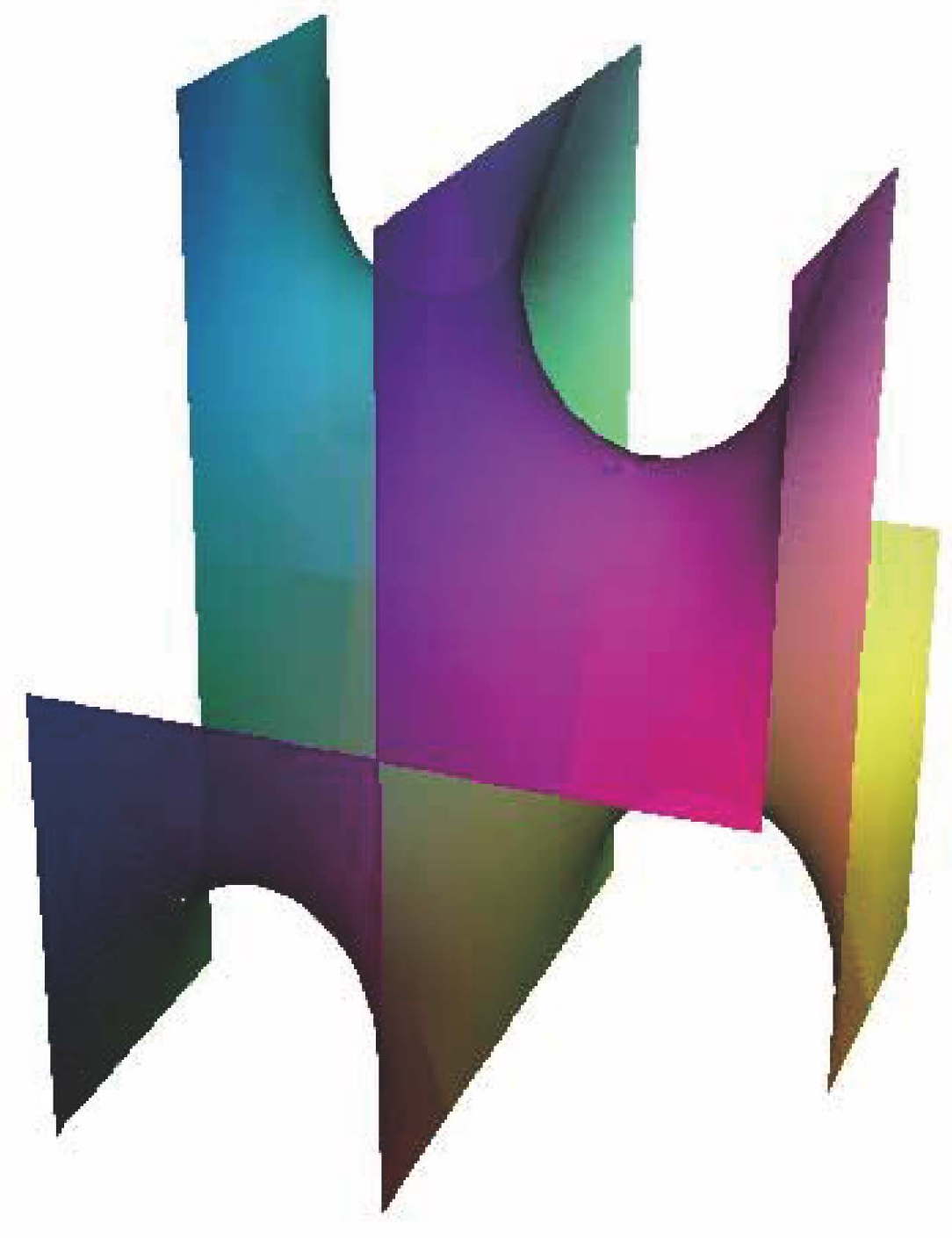} \includegraphics[scale=.4]{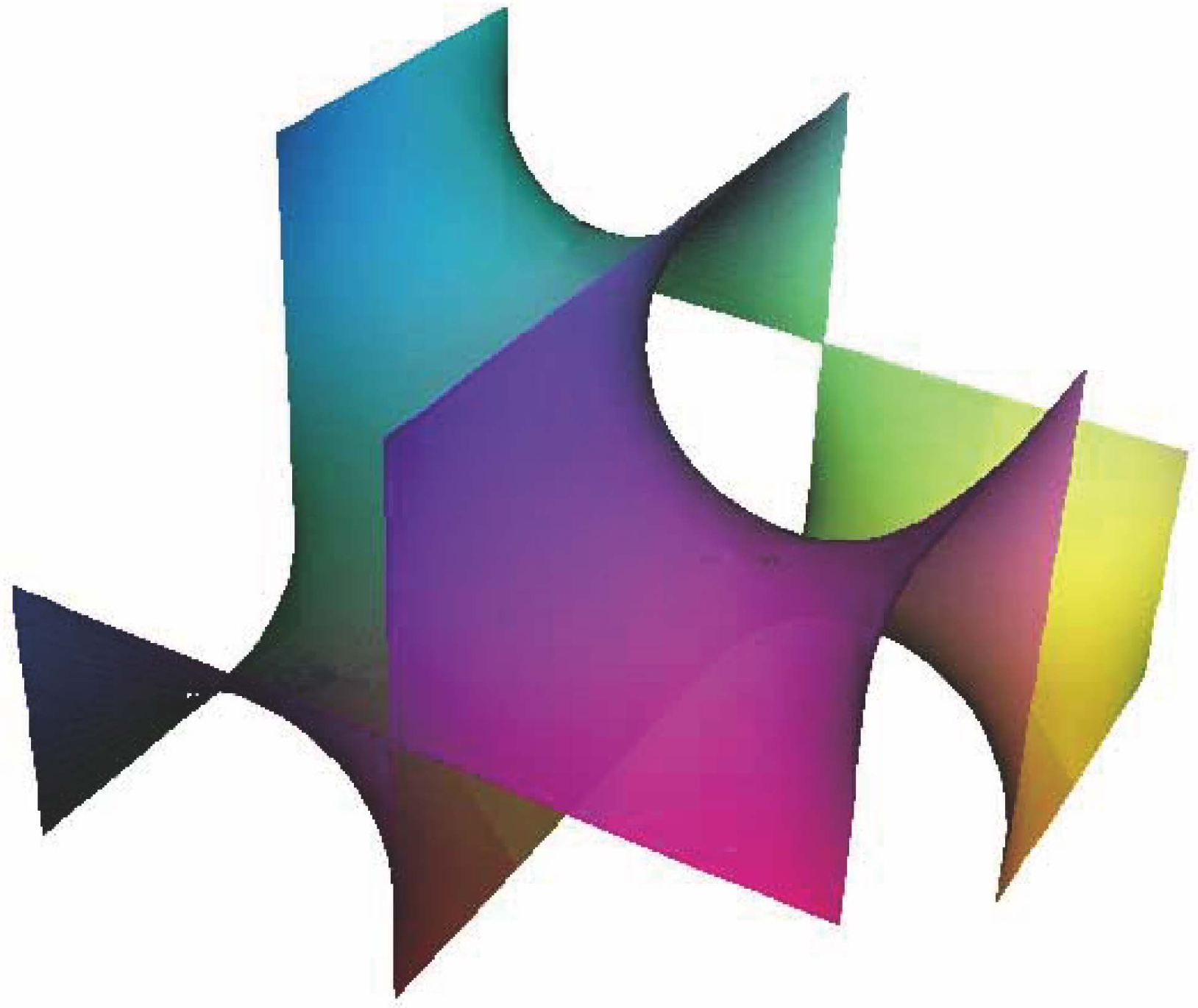}
\caption{Surfaces $M_a$ of the tD-family, with $a=2.85$ (Morse index $2$), and $a=14$ (Morse index $1$).}
\label{fig:tDfamily}
\end{figure}
\begin{proof}[Proof of Theorem~B]
As we have observed in Remark~\ref{thm:remLinIndip}, the linear independence assumption of Theorem~\ref{thm:IFT} is always satisfied in the case of all
minimal surfaces of the given families. This implies that equivariantly nondegenerate surfaces correspond precisely to minimal surfaces with nullity equal to $3$. Each such surface belongs to a unique smooth family of triply periodic minimal
surfaces, parameterized by isometry classes of flat metrics in $\mathds T^3$ (a $6$-dimensional space), by
Theorem~\ref{thm:IFT}. If we consider isometry classes of flat metrics with fixed volume, which are thus pairwise non-homothetic, we get a smooth $5$-parameter family of pairwise non-homothetic triply periodic minimal surfaces.
\end{proof}
\begin{remark}
By the proof of Theorem A, Examples \ref{exa:rPD} and \ref{exa:tPtD}, and Theorem B,
we can make the following observation.
Schwarz P-surface has nullity equal to $3$, and it is contained both in the rPD-family and in the tP-family,
which are two one-parameter families of triply periodic minimal surfaces.
Near the P-surface, these two families consist of surfaces that are pairwise non-homothetic.
On the other hand, by Theorem~B, the P-surface belongs to a (unique up to homotheties)
smooth locally rigid $5$-parameter family of pairwise non-homothetic triply periodic minimal surfaces.
By the local rigidity, such $5$-parameter family must contain (a portion of) the rPD-family and the tP-family, and it would be very interesting to have a geometric description of this $5$-parameter family.
An analogous situation occurs for the Schwarz D-surface, which has nullity equal to $3$, and it is contained both in the rPD-family and in the tD-family.
\end{remark}
\begin{example}[tCLP-family  --- Figure~\ref{fig:tCLPfamily}]\label{exa:tCLP}
For $a \in\left]-2,2\right[$, let $M_a$ be a hyperelliptic Riemann surface of genus $3$
defined by $w^2=z^8+a z^4+1$. Let $f$ be a conformal
minimal immersion given by
\[f(p)=\Re \int^p_{p_0} \big(1-z^2,\,i(1+z^2),\,2 z\big)^\mathrm{t}\, \dfrac{d z}{w}.\]
$f(M_a)$ is called tCLP family.
For all $a$, the minimal surface $M_a$ has constant nullity equal to $3$ and Morse index equal to $3$, see \cite[Main Theorem~4]{E-S}.
\begin{figure}
\includegraphics[scale=.4]{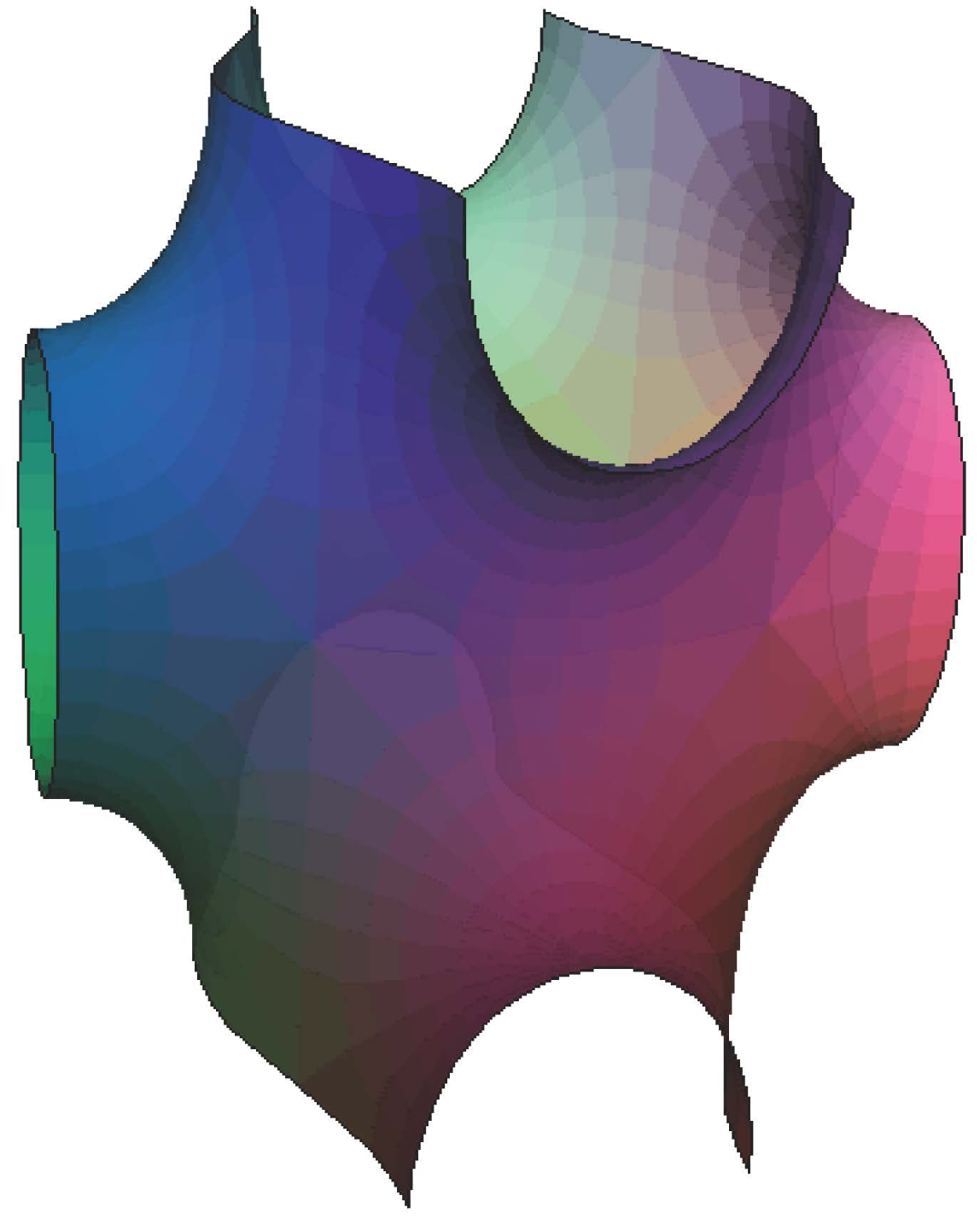}\quad\includegraphics[scale=.4]{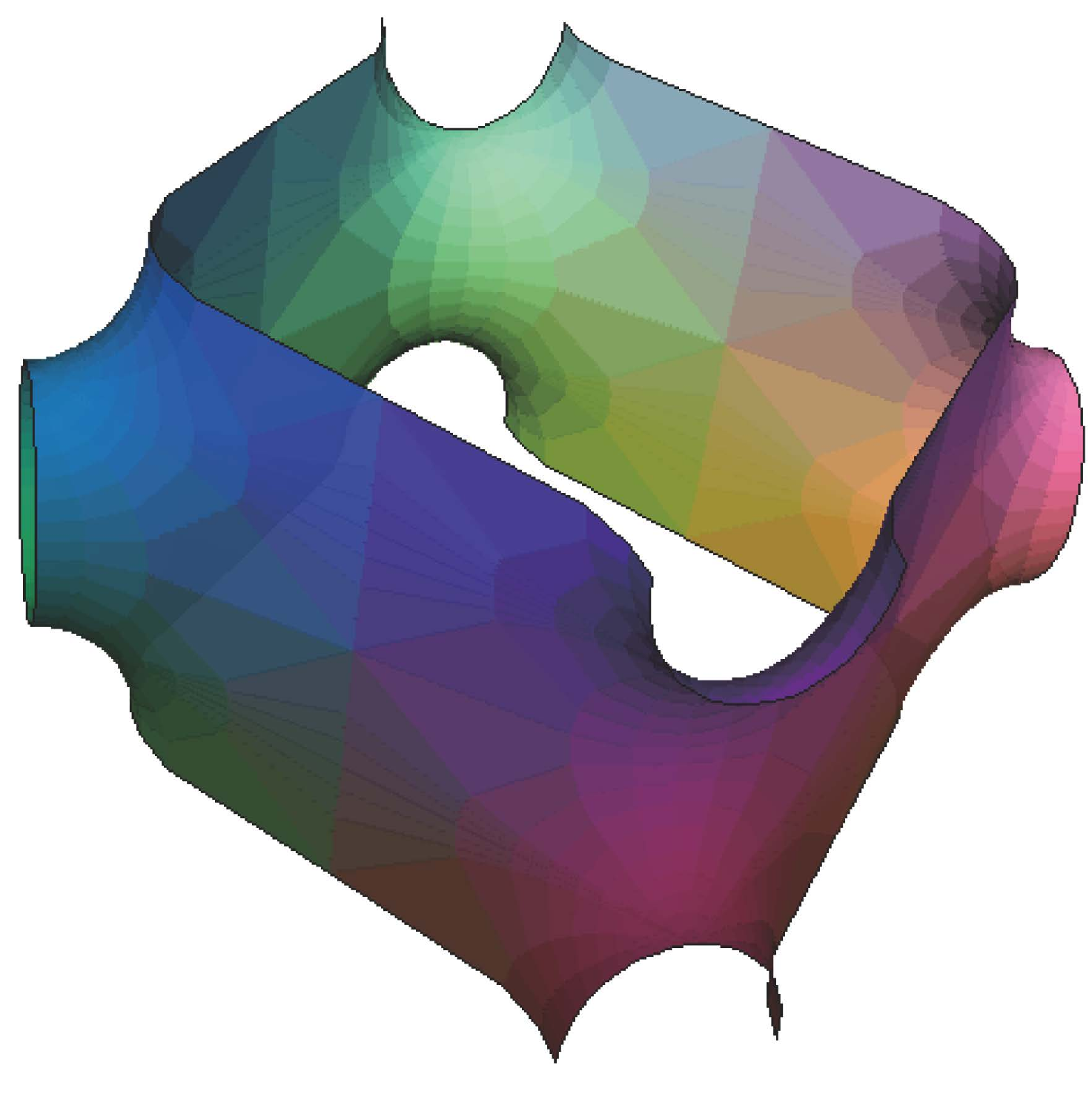}
\caption{Surfaces $M_a$ of the tCLP-family, with $a=0$ (also called the \emph{Schwarz CLP-surface}) and with
$a=1.96$. All the surfaces of this family have Morse index equal to $3$ and nullity equal to $3$.}
\label{fig:tCLPfamily}
\end{figure}
\end{example}

\begin{example}[associate family of Schwarz P-surface]\label{exa:PDG}
Let $M$ be a hyperelliptic Riemann surface of genus $3$
defined by $w^2=z^8+14z^4+1$.
Then the Schwarz P surface is given by
\[p\longmapsto \Re \int_{p_0}^p (1-z^2,\,i(1+z^2),\,2z)^t\dfrac{dz}{w},\]
and the Schwarz D surface is given by
\[ p\longmapsto \Re \int_{p_0}^p i (1-z^2,\,i(1+z^2),\,2z)^t\dfrac{dz}{w}.\]
The general associate surface of the P surface is given by
\begin{equation}\label{associate}
p\longmapsto \Re \int_{p_0}^p e^{i\theta} (1-z^2,\,i(1+z^2),\,2z)^t\dfrac{dz}{w}, \qquad \theta \in \mathds{R}.
\end{equation}
It is known that there is a unique $\theta \in\left ]0, \pi/2\right[$ such that \eqref{associate} gives a triply periodic minimal surface in $ \mathds{R}^3$.
Actually, $\theta \approx0.907313(=51.9852^\circ)$ gives this surface, which is called the Schoen's gyroid (\cite{A.Schoen1970}).
Since these three surfaces have the same Riemann metric:
\[
ds^2=\frac{(1+|z|^2)^2}{|w|^2}|dz|^2,
\]
and thus the same Jacobi operator:
\begin{equation}\label{jacobi}
J=\Delta - 2K
=\frac{4|w|^2}{(1+|z|^2)^4}
\Biggl((1+|z|^2)^2\frac{\partial^2}{\partial z\partial \overline{z}}+2
\Biggr),
\end{equation}
they also have the same nullity and the same Morse index.
Since Schwarz $P$ and $D$ surfaces have nullity $3$ and Morse index $1$ because they are volume-preserving stable (Ross \cite{Ross1992}), the Schoen's gyroid also
has nullity $3$ and Morse index $1$.
\end{example}

\begin{proof}[Proof of Theorem~C]
As the proof of Theorem~B,
it follows immediately from Theorem~\ref{thm:IFT}, since  Schoen's gyroid
has nullity equal to $3$ (Example \ref{exa:PDG}).
\end{proof}

\section{Remarks on the geometry of triply periodic minimal surfaces in the bifurcation branches}
\label{sec:remarks}
In general, Theorem \ref{thm:mainbiftheorem} does not imply\footnote{%
An instructive example of bifurcation by Morse index jump in geometric variational problems that does not produce new solutions is discussed in \cite[Section~2.6]{KoiPalPic14}, in the context of constant mean curvature surfaces.}
 the existence of essentially new triply periodic minimal surfaces, because there we do not assume that the flat metrics $g_{\Lambda_{(s)}}$
are pairwise non homothetic near $s=0$, see Remark~\ref{thm:remgenuinebifurcation}.

In this section we will discuss this question at each bifurcation instant that we obtained in
Theorem~A, and we will determine at which bifurcations instants one obtains the existence of new examples of triply periodic minimal surfaces. Every bifurcation instant given in Theorem~A corresponds to an equivariantly degenerate minimal embedding whose Jacobi operator has kernel spanned by the three
Killing--Jacobi fields (Definition~\ref{thm:defequivnondegeneracy}), and one additional Jacobi field which is not Killing. Such a Jacobi field gives a first order approximation of the bifurcating branch, and it will be used in
our discussion.
The exact equations of these Jacobi fields and how to derive them will be discussed in a forthcoming paper, and will be omitted here.
\subsection{rPD-family}
First, let us look at the bifurcation instants along the rPD-family $M_a$, $a\in\left]0,+\infty\right[$ (Example \ref{exa:rPD}).
Set
$$
A\colon=A(a)\colon=\frac{1}{\sqrt{3}a}\int_0^1\frac{1+a^2t^2}{\sqrt{t(1-t^3)(a^3t^3+\frac{1}{a^3})}}\;dt,
$$
$$
C \colon=C(a) \colon=4\int_0^1\frac{t}{\sqrt{t(1-t^3)(a^3+\frac{t^3}{a^3})}}\;dt.
$$
Then, the lattice is
\[
  \Lambda = \left(
    \begin{array}{ccc}
     3A & 3A & 4A \\
     \sqrt{3} A & - \sqrt{3}A & 0 \\
      0 & 0 & C
    \end{array}
  \right).
\]
$C$ is the height of the lattice, and $A$ is a certain fixed constant times the length of the edge of the triangle (see Figure  \ref{fig:rPdfamily}).
Hence, the ratio $A/C$ determines the lattice (up to homothety).
Figure \ref{fig:lattice-rPD} represents the ratio $A/C$ as a function of $a$,
and $a=a_1\approx 0.494722$ gives the minimum.
It shows that there exist positive constants $\epsilon_1$ and $\epsilon_2$,
a strictly monotone-increasing function $\delta\colon[0, \epsilon_1[ \to [0, \epsilon_2[$ with
$\delta(0)=0$ and such that
the lattice $\Lambda_{(a_1-\epsilon)}$ is a homothety of the lattice $\Lambda_{(a_1+\delta(\epsilon))}$.

Denote by $X(c)$ the surface $f(M_{a_1+c})$.
Since $A(a)$ and $C(a)$ are increasing functions of $a$ near the bifurcation instant $a=a_1\approx 0.494722$ (see Figure \ref{fig:rPD-AC}),
the surfaces $X(c)$ are like the pictures in the upper row in Figure \ref{fig:rPD-bifurcation}.

Now, for $\epsilon \in\left]0, \epsilon_1\right[$,
reduce the surface $X(\delta(\epsilon))$ to $\displaystyle \frac{C(a_1-\epsilon)}{C(a_1+\delta(\epsilon))}$, and denote the new surface by $Y(-\epsilon)$.
Then the lattice of $Y(-\epsilon)$ is the same as the lattice of $X(-\epsilon)$, but
the surfaces $X(-\epsilon)$ and $Y(-\epsilon)$ are not congruent to each other (see Figure \ref{fig:rPD-bifurcation}).

Similarly, for $\epsilon \in\left]0, \epsilon_1\right[$,
expand the surface $X(-\epsilon)$ to $\displaystyle \frac{C(a_1+\delta(\epsilon))}{C(a_1-\epsilon)}$, and denote the new surface by $Y(\delta(\epsilon))$.
Then the lattice of $Y(\delta(\epsilon))$ is the same as the lattice of $X(\delta(\epsilon))$,
but the surfaces $X(\delta(\epsilon))$ and $Y(\delta(\epsilon))$
are not congruent to each other (see Figure \ref{fig:rPD-bifurcation}).

One can show the nodal lines of the zero-eigenfunction at $a=a_1$ are exactly the boundary triangles in Figure \ref{fig:rPD-bifurcation}, and since the ``essential'' dimension of the zero eigenspace is one,
it seems that the surfaces $Y(c)$ give the bifurcation branch from the instant $a=a_1$.
However, they are homotheties of the original surfaces in the rPD family.

For $\epsilon \in\left]0, \epsilon_1\right[$, $X(-\epsilon)$ has index $2$ and nullity $3$,
$X(\delta(\epsilon))$ has index $1$ and nullity $3$,
$Y(-\epsilon)$ has index $1$ and nullity $3$,
$Y(\delta(\epsilon))$ has index $2$ and nullity $3$.
Hence, this bifurcation is a transcritical bifurcation.

\begin{figure}
\includegraphics[scale=.8]{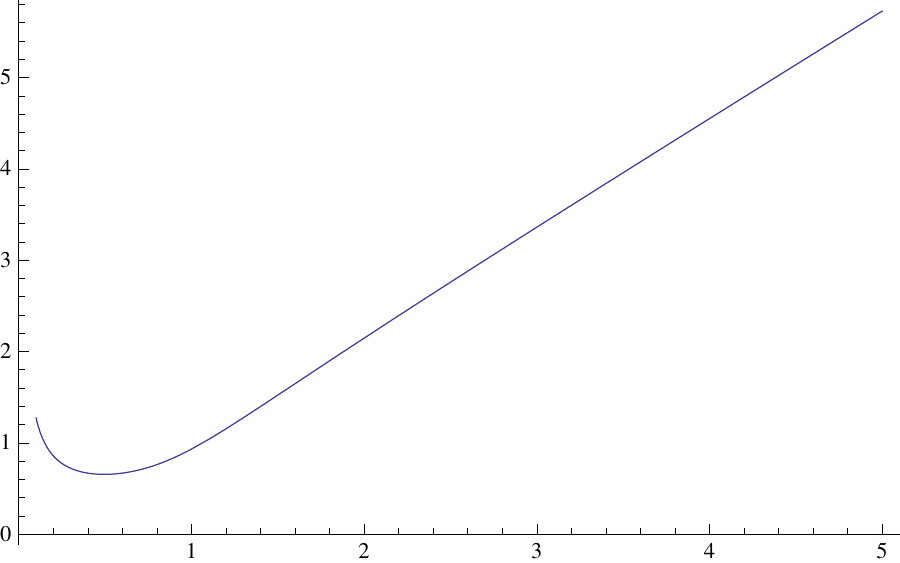}
\caption{%
The horizontal axis represents $a$, while the vertical axis indicates the ratio $A/C$ for the rPD family. The minimum of $A/C$ is attained at $a=a_1\approx 0.494722$.}
\label{fig:lattice-rPD}
\end{figure}

\begin{figure}
\includegraphics[scale=1.3]{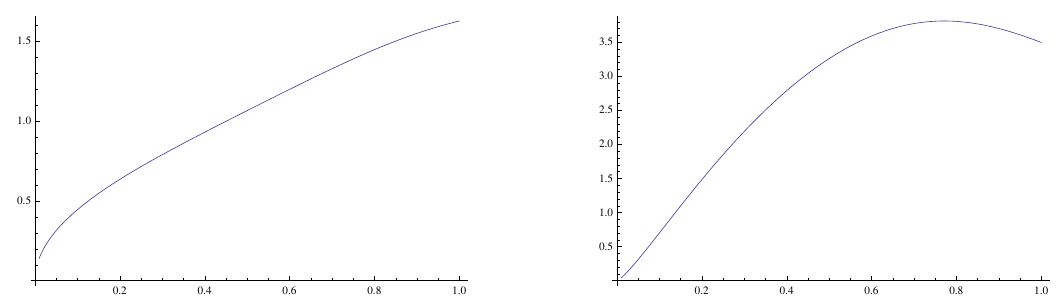}
\caption{The horizontal axis represents $a$, while the vertical axis indicates $A$ (left) and $C$ (right) for the rPD family.}
\label{fig:rPD-AC}
\end{figure}

\begin{figure}
\includegraphics[scale=.8]{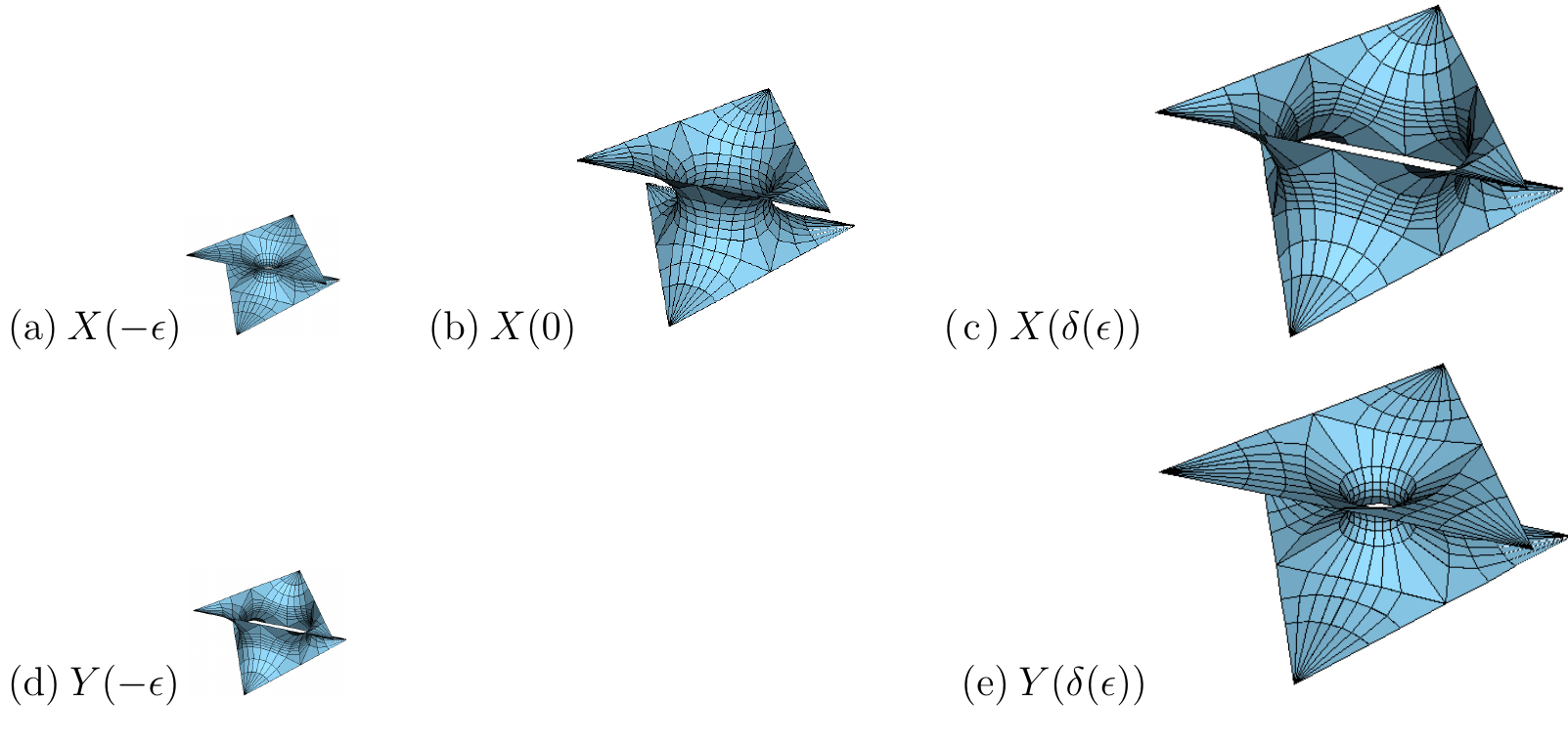}
\caption{Bifurcation at $a=a_1$ of the rPD family.
Each picture shows a half period of the corresponding triply periodic minimal surface.
The surfaces in the upper row belong to the rPD family.
The surfaces in the lower row belong to the bifurcation branch, and they are homothetic to surfaces in the rPD family.}
\label{fig:rPD-bifurcation}
\end{figure}

On the other hand, at $a=a_2 = (a_1)^{-1}\approx 2.02133$, the ratio $A/C$ is strictly monotone (Figure~\ref{fig:lattice-rPD}). This implies that the bifurcation branch contains triply periodic minimal surfaces
that are not homothetic to any other surface in the five families given in \S~\ref{sec:applications}
(Examples \ref{exa:H}---\ref{exa:tCLP}).
Moreover, we can show that the nodal lines of the zero eigenfunctions at $a=a_2$ are planar geodesics that connect each vertex of each triangle with the middle point of a side of a triangle
(Figure~\ref{fig:rPD+L}, the left picture in Figure~\ref{fig:rPD+lines}).
Remarkably, the sides of the triangles are not nodal lines, which suggests that
near $a=a_2$, our new triply periodic minimal surfaces are like the right picture in Figure \ref{fig:rPD+lines}.
It would be interesting to determine the symmetries of the new surfaces.

\begin{figure}
\includegraphics[scale=1.0]{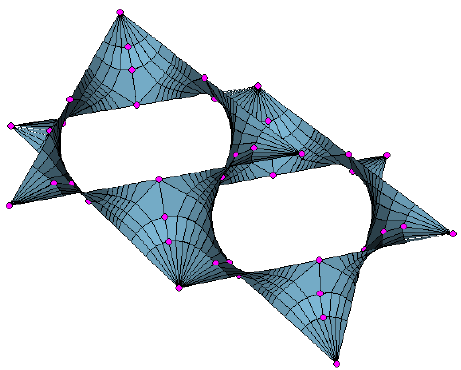}
\caption{One period of a surface in the rPD-family with the planar geodesics}.
\label{fig:rPD+L}
\end{figure}

\begin{figure}
\includegraphics[scale=.7]{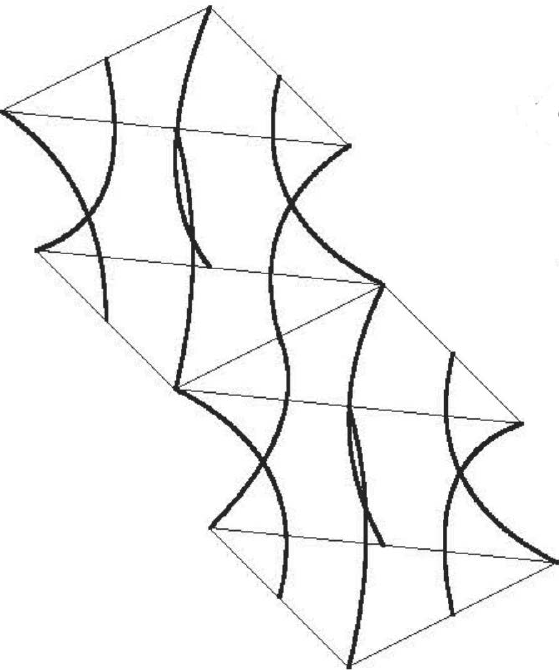}
\quad
\includegraphics[scale=.7]{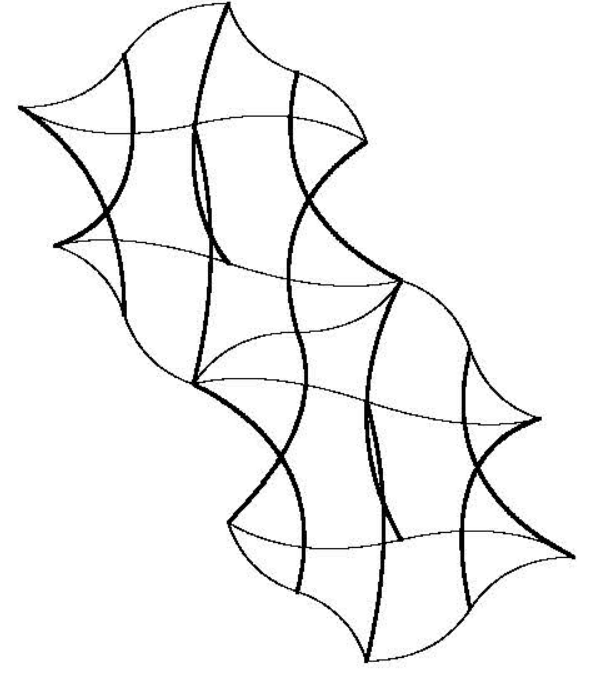}
\caption{Left: one period of a surface in the rPD-family with the planar geodesics,
Right: variation from a surface ($a=a_2$) in the rPD family with the zero eigenfunction as variation vector filed}.
\label{fig:rPD+lines}
\end{figure}
\subsection{H-family}
By a similar fashion, we find the lattice of the H-family.
Set
\begin{multline*}
B := \sqrt{3} \int_0^1
\dfrac{1-t^2}{\sqrt{t(t^3+a^3)(t^3+\frac{1}{a^3})}}\mathrm dt
\\+ 4\int_{\frac{1}{2}}^1 \dfrac{x}{\sqrt{
(a^3+\frac{1}{a^3}+6x-8x^3)(1-x^2) }}\mathrm dx,
\end{multline*}
$$
D := 8\int_0^1 \dfrac{t}{\sqrt{t(t^3+a^3)(t^3+\frac{1}{a^3})}} dt.
$$
Then, the lattice is
\[
  \Lambda = \left(
    \begin{array}{ccc}
      \frac{\sqrt{3}}{2}B &0 & 0 \\
      \frac{B}{2} & B & 0 \\
      0 & 0 & D
    \end{array}
  \right).
\]
The ratio $B/D$ determines the lattice (up to homothety).
Figure \ref{fig:lattice-H} represents the ratio $B/D$ as a function of $a$,
and $a=a_0\approx 0.49701$ gives the minimum of $B/D$.
Moreover, we can show that the nodal lines of the zero eigenfunctions at $a=a_0$ are exactly the triangles indicated in Figure~\ref{fig:H-family-T}.
And so, arguing as the case of the rPD-family, we conjecture that
the bifurcation we obtained in Theorem~A at $a=a_0$ for the H-family gives only homotheties of the surfaces in the original H-family.

\begin{figure}
\includegraphics[scale=1.0]{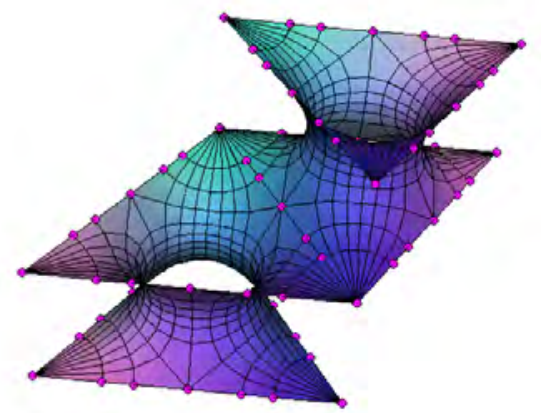}
\caption{One period of a surface in the H-family}.
\label{fig:H-family-T}
\end{figure}
\subsection{tP-family and tD-family}
By a similar way, we find the lattice of the tP-family.
Set
\[
E = 2 \int^{1}_{0}  \dfrac{1-t^2}{\sqrt{ t^8+a t^4+1}}dt + 4 \int_{0}^1
\dfrac{dt }{\sqrt{ 16 t^4-16 t^2+2+a}},
\]
\[
F= 8 \int_0^1 \dfrac{t}{\sqrt{t^8+a t^4+1}} dt.
\]
Then, the lattice is
\[
  \Lambda = \left(
    \begin{array}{ccc}
    E &0 & 0 \\
      0 & E & 0 \\
      0 & 0 & F
    \end{array}
  \right).
\]
The ratio $E/F$ determines the lattice (up to homothety).
Figure \ref{fig:lattice-tP} represents $E/F$ as a function of $a$.
By the same reason as the case of the rPD-family, we conjecture that
the bifurcation we obtained in Theorem~A at $a=a_2\approx 28.7783$ for the tP-family gives only homotheties of the surfaces in the original tP-family.
However, we conclude
 that the bifurcation at $a=a_1\approx 7.40284$ for the tP-family give
triply periodic minimal surfaces
that are not homothetic to any other surface in the five families given in \S~\ref{sec:applications}
(Examples \ref{exa:H}---\ref{exa:tCLP}).
As for the tD-family, the situation is totally analogous.

\begin{figure}
 \includegraphics[scale=.8]{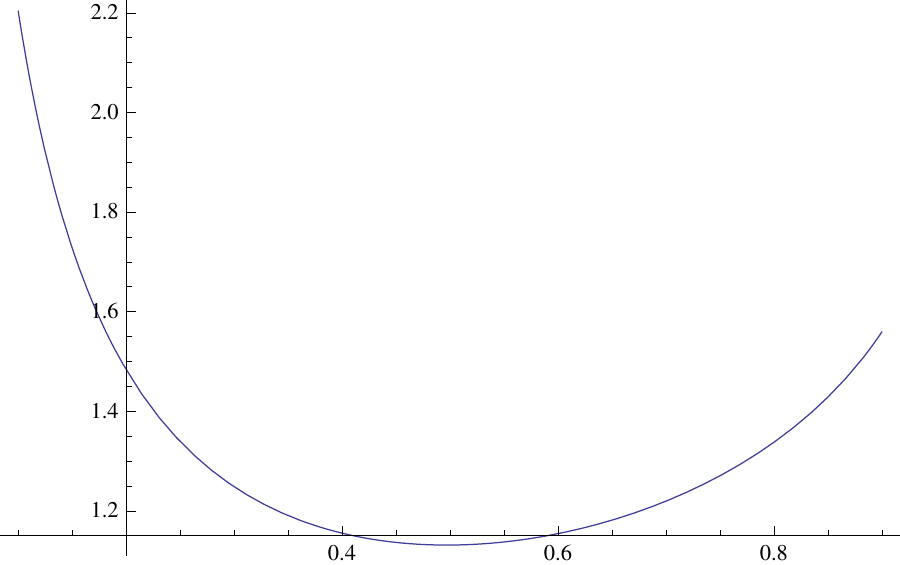}
\caption{The horizontal axis represents $a$, while the vertical axis indicates the ratio $B/D$ for the H-family. The minimum of $B/D$ is attained at $a=a_0\approx 0.49701$.}
\label{fig:lattice-H}
\end{figure}
\begin{figure}
\includegraphics[scale=.7]{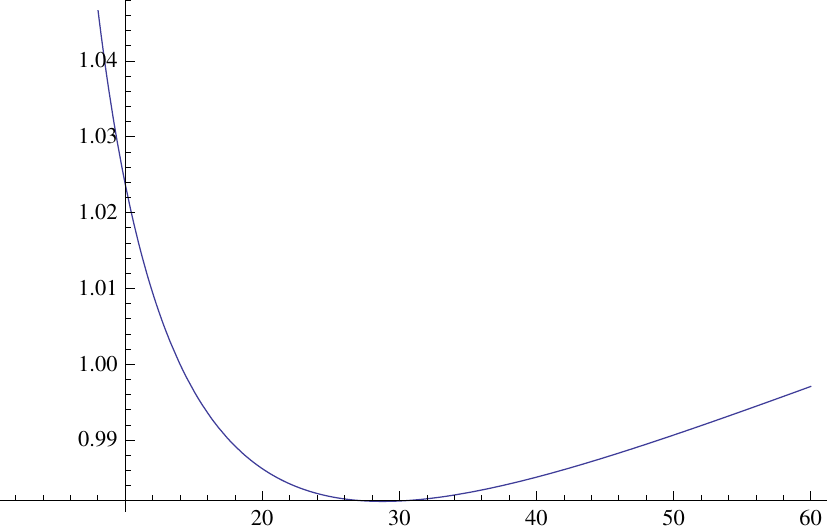}
\quad
\includegraphics[scale=.6]{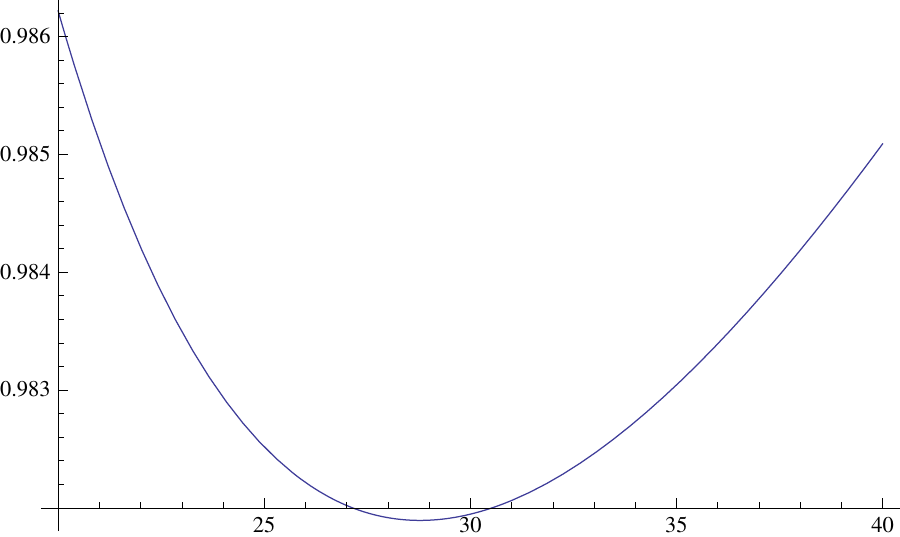}
\caption{
The horizontal axis represents $a$, while the vertical axis indicates the ratio $E/F$ for the tP-family. The minimum of $E/F$ is attained at $a=a_2\approx 28.7783$.}
\label{fig:lattice-tP}
\end{figure}

\appendix
\section{Auxiliary results}\label{app:technical}
\subsection{A divergence formula}
We will prove here formula \eqref{eq:compdivergence}, used in the proof of Proposition~\ref{thm:kapouleas}.
\begin{lemma}\label{thm:divergence}
Let $(\overline M,\overline{\mathbf g})$ be a Riemannian manifold, let $M\subset\overline M$ be
a compact submanifold (without boundary), with mean curvature vector field $\vec H$,
and let $K\in\mathfrak X(M)$ be a Killing field in $M$.
Denote by $K_M\in\mathfrak X(M)$ the vector field on $M$ obtained by orthogonal projection
of $K$. Then, $\mathrm{div}_M(K_M)=\overline{\mathbf g}(K,\vec H)$.
\end{lemma}
\begin{proof}
In order to compute $\mathrm{div}_M(K_M)$, let $\overline\nabla$ denote the Levi-Civita connection of $\overline{\mathbf g}$
and let $\nabla$ be the Levi--Civita connection of the induced metric on $M$. If $\mathcal S$ is the
second fundamental form of $M$, then for all pairs $X,Y\in\mathfrak X(M)$, one has
$\overline\nabla_XY=\nabla_XY+\mathcal S(X,Y)$. Moreover, differentiating in the direction $X$ the equality $\overline{\mathbf g}(K_M,Y)=\overline{\mathbf g}(K,Y),$ we get:
\begin{equation}\label{eq:lemA}
\overline{\mathbf g}\big(\nabla_XK_M,Y\big)+\overline{\mathbf g}\big(K_M,\nabla_XY\big)
=\overline{\mathbf g}\big(\overline\nabla_XK,Y\big)+\overline{\mathbf g}\big(K,\overline\nabla_XY\big).
\end{equation}
Substituting $\overline{\mathbf g}(K,\overline\nabla_XY)=\overline{\mathbf g}(K,\nabla_XY)+\overline{\mathbf g}\big(K,\mathcal S(X,Y)\big)$ in \eqref{eq:lemA} gives:
\begin{equation}\label{eq:lemB}
\overline{\mathbf g}\big(\nabla_XK_M,Y\big)=
\overline{\mathbf g}\big(\overline\nabla_XK,Y\big)+\overline{\mathbf g}\big(K,\mathcal S(X,Y)\big).
\end{equation}
Given $x\in M$, an orthonormal frame $e_1,\ldots,e_m$ of $T_xM$, and recalling that, since
$K$ is Killing, $\overline{\mathbf g}(\overline\nabla_{e_i}K,e_i)=0$
for all $i$, we get:
\[\mathrm{div}_M(K_M)=\sum_i\overline{\mathbf g}\big(\nabla_{e_i}K_M,e_i\big)=
\sum_i\overline{\mathbf g}\big(K,\mathcal S(e_i,e_i)\big)=\overline{\mathbf g}(K,\vec H).\qedhere\]
\end{proof}
\subsection{On bifurcation for families of Fredholm operators}
\label{sub:bifresult}
Let us recall briefly the precise statement of a well known bifurcation result for solutions of an equation of the form $F(\mathfrak x,\mu)=0$,
with $\mu\in[\mu_0-\delta,\mu_0+\delta]$, $\delta>0$, a real parameter and $F(\cdot,\mu):\mathfrak X\to\mathfrak Z$ is a continuous family of smooth maps
from the (open subset of a) Banach space $\mathfrak X$ to a Banach space $\mathfrak Z$.
Our basic references are the books \cite{Kato} and \cite{Kiel2012}.
\smallskip

Assume that:
\begin{itemize}
\item[(A)] $\mathfrak X$ is continuously embedded into $\mathfrak Z$, i.e., there exists a continuous injective linear map
$\mathfrak i\colon\mathfrak X\hookrightarrow\mathfrak Z$ (we will implicitely consider $\mathfrak X\subset\mathfrak Z$);
\item[(B)] $\left[\mu_0-\delta,\mu_0+\delta\right]\ni\mu\mapsto\mathfrak x_\mu\in\mathfrak X$ is a continuous map such that
$F(\mathfrak x_\mu,\mu)=0$ for all $\mu$;
\item[(C)] setting $A_\mu=\mathrm D_\mathfrak xF(\mathfrak x_\mu,\mu):\mathfrak X\to\mathfrak Z$, then for all $\mu$:
\begin{itemize}
\item[(C1)] $A_\mu$ is a Fredholm operator of index $0$;
\item[(C2)] $A_\mu\colon\mathfrak Z\to\mathfrak Z$ is closed as an unbounded linear operator with domain $\mathfrak X$;
\end{itemize}
\item[(D)] $0$ is an isolated eigenvalue of $A_{\mu_0}$.
\end{itemize}
Assumption (C1) implies that $\mathrm{Ker}(A_{\mu_0})$ is finite dimensional, while assumption (D) implies
that the generalized eigenspace $E_{\mu_0}=\bigcup_{k\ge1}\mathrm{Ker}(A_{\mu_0}^k)$ is also finite dimensional, see \cite[Section IV.5.4]{Kato}. For the spectral theory, one considers a complexification of the space $\mathfrak X$.

Deep results from perturbation theory, see \cite[Sections II.5.1 and III.6.4]{Kato}, imply that there exists $\delta'\in\left]0,\delta\right]$ and a continuous map of finite dimensional subspaces $\left[\mu_0-\delta',\mu_0+\delta'\right]\ni\mu\mapsto E_\mu\subset\mathfrak X$
such that for all $\mu$:
\begin{itemize}
\item $\mathrm{dim}(E_\mu)=\mathrm{dim}(E_{\mu_0})$;
\item $E_\mu$ is invariant by $A_\mu$.
\end{itemize}
Denote by $\overline A_\mu$ the restriction of $A_\mu$ to $E_\mu$, and $\epsilon_\mu$ denote the sign
of the determinant of $\overline A_\mu$:
\begin{equation}\label{eq:signet}
\epsilon_\mu=\begin{cases}+1,&\text{if $\mathrm{det}(\overline A_\mu)>0$};\\
\ \ 0,&\text{if $\mathrm{det}(\overline A_\mu)=0$};\\
-1,&\text{if $\mathrm{det}(\overline A_\mu)<0$}.
\end{cases}
\end{equation}
Note that $\epsilon_\mu=0$ only if $A_\mu$ is singular, because $\mathrm{Ker}(\overline A_\mu)=
\mathrm{Ker}(A_\mu)\cap E_\mu$.
\begin{biftheorem}
In the above situation, assume:
\begin{enumerate}
\item[(BT1)] for $\mu\in\left[\mu_0-\delta,\mu_0\right[\bigcup\left]\mu_0,\mu_0+\delta\right]$, the operator $A_\mu$ is nonsingular\footnote{i.e.,
$A_\mu\colon\mathfrak X\to\mathfrak Z$ is an isomorphism};
\item[(BT2)] $\epsilon_{\mu_0-\delta}\ne\epsilon_{\mu_0+\delta}$.
\end{enumerate}
Then, $(\mathfrak x_{\mu_0},\mu_0)$ is a bifurcation point for the equation $F(\mathfrak x,\mu)=0$, i.e.,
the closure of the set $\big\{(\mathfrak x,\mu):\mathfrak x\ne\mathfrak x_\mu,\ F(\mathfrak x,\mu)=0\big\}$ contains
$(\mathfrak x_{\mu_0},\mu_0)$.
\end{biftheorem}
\begin{proof}
See \cite[Theorem II.4.4]{Kiel2012}.
\end{proof}
Condition (BT2) in the above theorem is usually referred to by saying that $A_\mu$ has an \emph{odd crossing number}
at $\mu=\mu_0$.
\medskip

The bifurcation result for minimal embeddings proved in Theorem~\ref{thm:mainbiftheorem} employs
the above bifurcation criterion for Fredholm operators. Usually, the odd crossing number assumption (BT2) is
hard to verify, in that one has no explicit description of the perturbed eigenspaces $E_\mu$.
However, in some cases the following elementary observation simplifies the task.
\begin{remark}\label{thm:simplifyingobs}
Assume that $\left[\mu_0-\delta,\mu_0+\delta\right]\ni\mu\mapsto A_\mu,A'_\mu$ are continuous paths
of Fredholm operators of index $0$ from $\mathfrak X$ to $\mathfrak Z$, with $A_{\mu_0}=A'_{\mu_0}$,
and assume that for all $\mu\in\left[\mu_0-\delta,\mu_0\right[\bigcup\left]\mu_0,\mu_0+\delta\right]$, both $A_\mu$
and $A'_\mu$ are nonsingular.
Assume that for all $r>0$ there exists $\delta'>0$ and two continuous paths
of invertible operators in the ball\footnote{Recall that the set of Fredholm operators of index $0$ is an open subset
of the space of all bounded linear operator from $\mathfrak X$ to $\mathfrak Z$ endowed with the operator norm.}
$B(A_{\mu_0},r)$ centered at $A_{\mu_0}$ and of radius $r$
joining $A_{\mu_0-\delta}$ with $A'_{\mu_0-\delta'}$ and $A_{\mu_0+\delta}$ with $A'_{\mu_0+\delta'}$
respectively. Then $A_\mu$ has an odd crossing number at $\mu_0$ if and only
if $A'_\mu$ has an odd crossing number at $\mu_0$. This follows easily from the
fact that the sign function $\epsilon$, which is defined in a sufficiently small neighborhood of $A_{\mu_0}$,
is constant along continuous paths of invertible operators. This is
because the sign of the determinant does not change along  continous paths of invertible linear maps.
\end{remark}


\begin{thebibliography}{9}
\bibitem{AliPic10} \textsc{L. J. Al\'\i as and P. Piccione},
\emph{On the manifold structure of the set of unparametrized embeddings with low regularity},
Bull.\ Braz.\ Math.\ Soc.\ (N.S.) 42 (2011), no.\ 2, 171--183.
\bibitem{AndHydLid} \textsc{S. Andersson , S. T. Hyde , K. Larsson , S. Lidin}, \emph{Minimal surfaces and structures: from inorganic and metal crystals to cell membranes and biopolymers}, Chem. Rev., 1988, 88 (1),  221--242.
\bibitem{BetPicSic2013} \textsc{R. G. Bettiol, P. Piccione, G. Siciliano},
\emph{Deforming solutions of geometric variational problems with varying symmetry groups}, Transform.\ Groups \textbf{19} (2014), no.\ 4, 941--968.
\bibitem{eji2000} \textsc{N. Ejiri}, \emph{A differential-geometric Schottky problem, and minimal surfaces in tori}, Differential geometry and integrable systems (Tokyo, 2000), 101--144,
Contemp.\ Math., 308, Amer.\ Math.\ Soc., Providence, RI, 2002.
\bibitem{eji2013} \textsc{N. Ejiri}, \emph{A generating function of a complex Lagrangian cone in ${\bf H}^n$}, preprint.
\bibitem{E-S} \textsc{N. Ejiri and T. Shoda},
\emph{The Morse index of a triply periodic minimal surface}, preprint.
\bibitem{FisKoc87} \textsc{W. Fischer, E. Koch}, \emph{On 3-periodic minimal surfaces} Z. Kristallogr.\ 179: 31--52 (1987).
\bibitem{FogHyd99} \textsc{A. Fogden, S.T. Hyde}, \emph{Continuous transformations of cubic minimal surfaces},
Eur. Phys. J. B 7, 91--104 (1999).
\bibitem{SchFogHyd06} \textsc{A. Fogden G. E. Schr\"oder-Turk, S. T. Hyde}, \emph{Bicontinuous geometries
and molecular self-assembly: comparison of local curvature and global packing
variations in genus-three cubic, tetragonal and rhombohedral surfaces},
Eur.\ Phys.\ J. B. 54 (2006), 509--524.
\bibitem{Kap1}  \textsc{N. Kapouleas}, \emph{Constant mean curvature surfaces in Euclidean three-space},
Bull.\ Amer.\ Math.\ Soc.\ (N.S.) \textbf{17} (1987), no.\ 2, 318--320.
\bibitem{Kap2} \textsc{N. Kapouleas}, \emph{Complete constant mean curvature surfaces in Euclidean three-space}, Ann.\ of Math.\ (2) \textbf{131} (1990), no.\ 2, 239--330.
\bibitem{Karcher89} \textsc{H. Karcher}, \emph{The triply periodic minimal surfaces of Alan Schoen and their constant mean curvature companions}, Manuscripta Math.\ \textbf{64} (1989), no.\ 3, 291--357.
\bibitem{Kato} \textsc{T. Kato}, \emph{Perturbation theory for linear operators}, Reprint of the 1980 edition. Classics in Mathematics. Springer--Verlag, Berlin, 1995.
\bibitem{Kiel2012} \textsc{H. Kielh\"ofer}, \emph{Bifurcation theory. An introduction with applications to partial differential equations}, Second edition. Applied Mathematical Sciences, 156. Springer, New York, 2012.
\bibitem{KoiPalPic14} \textsc{M. Koiso, B. Palmer, P. Piccione}, \emph{Bifurcation and symmetry breaking of nodoids with fixed boundary}, preprint 2014, to appear in Advances in Calculus of Variations.
\bibitem{mpp}{\textsc{R. Mazzeo, F. Pacard \and D. Pollack}, \emph{Connected sums of constant mean curvature surfaces in Euclidean 3 space}. J. Reine Angew. Math. 536 (2001), 115--165.}
\bibitem{mp}{\sc R. Mazzeo \and F. Pacard}, \emph{Constant mean curvature surfaces with Delaunay ends.} Comm.\ Anal.\ Geom. 9 (2001), no.\ 1, 169--237.
\bibitem{mee90} \textsc{W. H. Meeks III}, \emph{The theory of triply periodic minimal surfaces},
Indiana Univ.\ Math.\ J. 39 (1990), no.\ 3, 877--936.
\bibitem{NagSmy75} \textsc{T. Nagano, B. Smyth}, \emph{Minimal varieties and harmonic maps in tori}, Comment.\ Math.\ Helv.\ 50 (1975), 249--265.
\bibitem{PerRos96} \textsc{J. P\'erez \and A. Ros}, \emph{The space of properly embedded minimal surfaces with finite
total curvature},  Indiana Univ.\ Math.\ J.\ \textbf{45} (1996), no.\ 1, 177--204.
\bibitem{Ross1992} \textsc{M. Ross}, \emph{Schwarz' $P$ and $D$ surfaces are stable},  Differential Geometry and its Applications\ \textbf{2} (1992), 179--195.
\bibitem{SchNes91} \textsc{H. G. von Schnering, R. Nesper}, \emph{Nodal surfaces of Fourier series: Fundamental invariants of structured matter}, Zeitschrift f\"ur Physik B Condensed Matter
1991, Volume 83, Issue 3,  407--412.
\bibitem{A.Schoen1970} \textsc{A. H. Schoen},  \emph{Infinite periodic minimal surfaces without self-intersections},
NASA Technical Note No. TN D-5541 (1970).
\bibitem{Trai08} \textsc{M. Traizet},  \emph{On the genus of triply periodic minimal surfaces},
J. Differential Geom.\ 79 (2008), no.\ 2, 243--275.
\bibitem{Whi} \textsc{B.~White}, \emph{The space of $m$-dimensional surfaces that are stationary for a parametric
elliptic functional}, Indiana Univ.\ Math.\ J.\ \textbf{36} (1987), 567--602.
\end{thebibliography}
\end{document}